\documentclass{article} 

\pdfoutput=1  %  for arxiv

\usepackage{fncylab,enumerate,wrapfig,placeins}
\usepackage[yyyymmdd,hhmmss]{datetime}

\usepackage{amsmath}

\usepackage{amsthm,amsfonts,amssymb, graphicx}

\usepackage{caption}
 \usepackage{textgreek,upgreek,bm}

 \usepackage{titlesec}

\usepackage{geometry}
\def\mindex#1{\index{#1}}

 \def\oc{\star}   % for optimal control
\def\ocp{*}   % for optimal parameter or function approximation

\newcommand{\qsaprobe}{{\scalebox{1.1}{$\upxi$}}}  %\upzeta \textphi
\newcommand{\bfqsaprobe}{{\scalebox{1.1}{$\bm{\upxi}$}}}  

\def\BE{{\cal B}}

\def\Tdiff{\mathcal{D}}

\def\ODEstate{\Uptheta} %\xi
\def\bfODEstate{\bm{\Uptheta}}

\def\odestate{\upvartheta}

\def\Nsam{N}

\def\fee{\upphi}

\def\Hor{\mathcal{T}} %H}
\def\hor{\uptau}   %\iota}

\def\elig{\zeta}

\def\uc{\underline{c}}

\def\uQ{\underline{Q}}

\def\disc{\gamma}

\newcommand{\bbblot}{\raise1pt\hbox{\vrule height .4ex width .4ex depth .05ex}}
\long\def\defbox#1{\framebox[.9\hsize][c]{\parbox{.85\hsize}{%
\parindent=0pt
\baselineskip=12pt plus .1pt      % STYLE 
\parskip=6pt plus 1.5pt minus 1pt % CHANGES
 #1}}}

\long\def\beginbox#1\endbox{\subsection*{}%
\hbox{\hspace{.05\hsize}\defbox{\medskip#1\bigskip}}%
\subsection*{}}

\def\endbox{}

 \def\archival#1{} %Notes I'd like to save

\def\llbracket{[\![}
\def\rrbracket{]\!]}

\def\FRAC#1#2#3{\genfrac{}{}{}{#1}{#2}{#3}}

\def\ddt{{\mathchoice{\FRAC{1}{d}{dt}}%
{\FRAC{1}{d}{dt}}%
{\FRAC{3}{d}{dt}}%
{\FRAC{3}{d}{dt}}}}

\def\ddtp{{\mathchoice{\FRAC{1}{d^{\hbox to 2pt{\rm\tiny +\hss}}}{dt}}%
{\FRAC{1}{d^{\hbox to 2pt{\rm\tiny +\hss}}}{dt}}%
{\FRAC{3}{d^{\hbox to 2pt{\rm\tiny +\hss}}}{dt}}%
{\FRAC{3}{d^{\hbox to 2pt{\rm\tiny +\hss}}}{dt}}}}

\def\ddyp{{\mathchoice{\FRAC{1}{d^{\hbox to 2pt{\rm\tiny +\hss}}}{dy}}%
{\FRAC{1}{d^{\hbox to 2pt{\rm\tiny +\hss}}}{dy}}%
{\FRAC{3}{d^{\hbox to 2pt{\rm\tiny +\hss}}}{dy}}%
{\FRAC{3}{d^{\hbox to 2pt{\rm\tiny +\hss}}}{dy}}}}

\def\half{{\mathchoice{\FRAC{1}{1}{2}}%
{\FRAC{1}{1}{2}}%
{\FRAC{3}{1}{2}}%
{\FRAC{3}{1}{2}}}}

\def\fourth{{\mathchoice{\FRAC{1}{1}{4}}%
{\FRAC{1}{1}{4}}%
{\FRAC{3}{1}{4}}%
{\FRAC{3}{1}{4}}}}

\newsavebox{\junk}
\savebox{\junk}[1.6mm]{\hbox{$|\!|\!|$}}

\def\limsup{\mathop{\rm lim{\,}sup}}

\def\argmin{\mathop{\rm arg{\,}min}}
\def\argmax{\mathop{\rm arg{\,}max}}

\def\Dds{\text{\rm F}}

\def\state{{\sf X}}

\def\ustate{{\sf U}} 
 
\def\zstate{{\sf Z}}

\def\bfmath#1{{\mathchoice{\mbox{\boldmath$#1$}}%
{\mbox{\boldmath$#1$}}%
{\mbox{\boldmath$\scriptstyle#1$}}%
{\mbox{\boldmath$\scriptscriptstyle#1$}}}}

\def\bfPhi{\bfmath{\Phi}}

\def\bfGamma{\bfmath{\Gamma}}

\def\bfmu{\bfmath{u}} 
 
\def\bfmx{\bfmath{x}}

\def\bfmN{\bfmath{N}}

\def\bfmX{\bfmath{X}}

\def\bfmY{\bfmath{Y}}

\def\bfmhhaY{\bfmath{\hhaY}} %\widehat{\widehat{Y}}}}
\def\bfmhhaY{\hbox to 0pt{$\widehat{\bfmY}$\hss}\widehat{\phantom{\raise 1.25pt\hbox{$\bfmY$}}}}

\def\haA{\widehat A}

\def\clE{{\cal E}}

\def\clG{{\cal G}}
\def\clH{{\cal H}}

\def\clP{{\cal P}}

\def\eqdef{\mathbin{:=}}

\def\Prob{{\sf P}}

\def\Expect{{\sf E}}

\def\lgmath#1{{\mathchoice{\mbox{\large #1}}%
{\mbox{\large #1}}%
{\mbox{\tiny #1}}%
{\mbox{\tiny #1}}}}

\def\Zero{{\mathchoice{\lgmath{\sf 0}}%
{\mbox{\sf 0}}%
{\mbox{\tiny \sf 0}}%
{\mbox{\tiny \sf 0}}}}

 \def\Tol{\text{\rm Tol}}

\def\ind{\bbbone}
  
 \def\epsy{\varepsilon}

\def\varble{\,\cdot\,}

\def\formtmp#1#2{{\vskip12pt\noindent\fboxsep=0pt\colorbox{#1}{\vbox{\vskip3pt\hbox to \textwidth{\hskip3pt\vbox{\raggedright\noindent\textbf{#2\vphantom{Qy}}}\hfill}\vspace*{3pt}}}\par\vskip2pt%
\noindent\kern0pt}}

\titleformat\subparagraph[runin]
                        {\normalfont\normalsize\bfseries}
                        {mypar}
                        {0pt}
                        {}{}
\titlespacing\subparagraph{0pt}%%             left margin spacing
                       {.1ex minus 0.2ex}%% spacing above the \subparagraph
                       {.75em}%%             horizontal offset from title (horizontal because the `runin` shape was used.)

\newenvironment{programcode}[1]{\ignorespaces\def\stmtopen##1{##1}%
\pagebreak[3]%
\formtmp{programcode}{#1}%
\endlinechar=-1\relax%
\nopagebreak[4]}{%
\noindent\textcolor{programcode}{\rule{\columnwidth}{1pt}}\vskip1pt\par\addvspace{\baselineskip}%
\endlinechar=13}

\newtheoremstyle{thm}{12pt}{15pt}%
     {\itshape}%         Body font
     {}%         Indent amount (empty = no indent, \parindent = para indent)
     {\bfseries}% Thm head font
     {}%        Punctuation after thm head
     {0pt}%{8pt}%     Space after thm head (\newline = linebreak)
     {\thmname{#1}\thmnumber{ #2.}\thmnote{ \textbf{(#3)}}\quad}% Thm head spec

\def\barclE{\bar{\clE}}

\def\barf{{\overline {f}}}
\def\barg{{\overline {g}}}

\def\barv{{\overline {v}}}

\def\barL{{\bar{L}}}

\def\barsigma{\overline{\sigma}}

{\end{list}}

\def\ass(#1:#2){(#1\ref{#1:#2})}

\def\ritem#1{
\item[{\sf \ass(\current_model:#1)}]
}

\newenvironment{recall-ass}[1]{% 
\begin{description}
\def\current_model{#1}}{
\end{description}
}

\def\sq{\hbox{\rlap{$\sqcap$}$\sqcup$}}
\def\qed{\ifmmode\sq\else{\unskip\nobreak\hfil
\penalty50\hskip1em\null\nobreak\hfil\sq
\parfillskip=0pt\finalhyphendemerits=0\endgraf}\fi}

\newcommand{\blot}{\vrule height 1.1ex width .9ex depth -.1ex }
\def\qedb{\ifmmode\blot\else{\vspace{-.2cm}\unskip\nobreak\hfil
\penalty50\hskip1em\null\nobreak\hfil\blot
\parfillskip=0pt\finalhyphendemerits=0\endgraf}\fi}

\newtheoremstyle{example}{15pt}{20pt}%
     {}%         Body font
     {}%         Indent amount (empty = no indent, \parindent = para indent)
     {\bfseries}% Thm head font
     {}%        Punctuation after thm head
     {1pt}%{8pt}%     Space after thm head (\newline = linebreak)
     {\thmname{#1}\thmnumber{ #2.}~\thmnote{\textit{\textbf{#3}}}%
     \\[.15cm]\unskip\nobreak}% Thm head spec

\theoremstyle{example}
\newtheorem{exmp}{Example}[section]

\theoremstyle{example}
\newcounter{rmnum}
\newenvironment{romannum}{\begin{list}{{\upshape (\roman{rmnum})}}{\usecounter{rmnum}
\setlength{\leftmargin}{18pt}
\setlength{\rightmargin}{8pt}
\setlength{\itemindent}{2pt}
}}{\end{list}}

\newcounter{anum}

\newcommand{\field}[1]{\mathbb{#1}}

\def\Re{\field{R}}

\def\Prob{{\sf P}}

\def\Expect{{\sf E}}

\def\transpose{{\intercal}}

\def\st{\text{\rm s.t.\,}}

\def\argmin{\mathop{\rm arg\, min}}
\def\ind{\hbox{\large \bf 1}}

\def\trace{\hbox{\rm trace\,}}  
\def\epsy{\varepsilon}
\def\varble{\,\cdot\,}

\def\haY{\widehat{Y}}

\def\hhaY{\hbox to 0pt{$\haY$\hss}\widehat{\phantom{\raise 1.25pt\hbox{Y}}}}

\def\haA{\widehat A}

\def\haY{\widehat Y}

\def\bfPhi{\bfmath{\Phi}}

\def\fudgeBE{\varrho^\tinyBEsymbol}

\def\qsaDyn{\text{H}}

\def\grave{\blacktriangle}
\def\Kern{{\Bbbk}}

\def\tinyBEsymbol{\upvarepsilon}  %\bullet  \flat
\def\zBE{z^\tinyBEsymbol}
\def\barzBE{\bar{z}^\tinyBEsymbol}
\def\clEBE{\clE^\tinyBEsymbol}
\def\barclEBE{\bar{\clE}^\tinyBEsymbol}

\def\clEBEvar{\clE^{\tinyBEsymbol \text{\tiny\rm var}}}
\def\barclEBEvar{\bar{\clE}^{\tinyBEsymbol \text{\tiny\rm var}}}

\def\kappaBE{\kappa^\tinyBEsymbol}

\def\regBCQ{\mathcal{R}}

\def\DPLP{DPLP}

 \graphicspath{{figures/}}

 \theoremstyle{thm}
\newtheorem{theorem}{Theorem}[section]
\newtheorem{lemma}[theorem]{Lemma}
\newtheorem{proposition}[theorem]{Proposition}
\newtheorem{corollary}[theorem]{Corollary}

\usepackage[usenames,dvipsnames]{color}
\usepackage{color}
 \definecolor{programcode}{gray}{0.9}

\definecolor{MyDarkBlue}{cmyk}{0.5,0.1,0,0.9}

\usepackage[colorlinks,%
linkcolor=BrickRed,% 
filecolor =BrickRed,%
citecolor=RoyalPurple,%
]{hyperref}

\newlength{\noteWidth}
\long\def\notes#1{\ifinner
{\footnotesize #1}
\else 
\marginpar{\parbox[t]{\noteWidth}{\raggedright\footnotesize#1}}
\fi\typeout{#1}}

 \def\notes#1{\typeout{check notes!!!}}   %  For final version

\def\choreS#1{\notes{\color{BrickRed}For SPM:  #1}}
\def\choreP#1{\notes{\color{RoyalPurple}For PGM:  #1}}

\def\saveForBook#1{}

\usepackage{cleveref}

\Crefname{corollary}{Corollary}{Corollaries}
\Crefname{eqnarray}{eq.}{eqs.}
\Crefname{equation}{eq.}{eqs.}

\Crefname{figure}{Fig.}{Figs.}
\Crefname{tabular}{Tab.}{Tabs.}
\Crefname{table}{Tab.}{Tabs.}
\Crefname{proposition}{Prop.}{Propositions}
\Crefname{theorem}{Thm.}{Thms.}
\Crefname{definition}{Def.}{Defs.} 
\Crefname{section}{Section}{Sections}
\Crefname{lemma}{Lemma}{Lemmas}
\Crefname{assumption}{Assumption}{Assumptions}

\title{{\LARGE Convex Q-Learning}
\\[.5em]
\Large
Part 1:  Deterministic Optimal Control 
}
\author{Prashant G. Mehta and Sean P. Meyn% <-this % stops a space 
\thanks{PGM is with the Coordinated
  Science Laboratory and the Department of Mechanical Science and
  Engineering at the University of Illinois at Urbana-Champaign
  (UIUC);   SPM is with the Department of Electrical and Computer Engineering, University of Florida, Gainesville, FL 32611,   and Inria International Chair.   Financial support from ARO award W911NF1810334
and National Science Foundation award  EPCN 1935389
is gratefully acknowledged.}%
}

\date{}

\begin{document}
 
\maketitle

\begin{abstract}

It is well known that the extension of Watkins' algorithm to general  function approximation settings is challenging:  does the ``projected Bellman equation'' have a solution? If so, is the solution useful in the sense of generating a good policy?  And, if the preceding questions are answered in the affirmative,  is the algorithm consistent?  These questions are unanswered even in the special case of  Q-function approximations that are linear in the parameter.  The challenge seems paradoxical, given the long history of convex analytic approaches to dynamic programming.  

The paper begins with a brief survey of linear programming approaches to optimal control, leading to a particular `over parameterization’ that lends itself to applications in reinforcement learning.  The main conclusions are summarized as follows:
\begin{romannum}

\item   
The new class of convex Q-learning   algorithms is introduced based on the convex relaxation of the Bellman equation.  Convergence is established under general conditions, including a linear function approximation for the Q-function.

\item    A batch implementation appears similar to the famed DQN algorithm (one engine behind AlphaZero).  It is shown that in fact the algorithms are very different: while convex Q-learning solves a convex program that approximates the Bellman equation,  theory for DQN is no stronger than for Watkins’ algorithm with function approximation:  (a) it is shown that both seek solutions to the same fixed point equation, and (b)  the ``ODE approximations'' for the two algorithms coincide,  and little is known about the stability of this ODE.    

%
%\item
% Of particular interest is the batch implementation, which appears similar to the famed DQN algorithm (the engine behind AlphaZero).  It is shown that in fact the algorithms are very different:   
%convex Q-learning solves a convex program that approximates the Bellman equation.  It is shown that DQN, if it converges,  will solve the same fixed point equation that arises in Watkins algorithm with function approximation.    Moreover, the ODE approximations of DQN and the classical algorithm coincides, which means that theory for both prior algorithms is currently lacking. 
\end{romannum}
These results are obtained for deterministic nonlinear systems with total cost criterion.  
Many extensions are proposed, including kernel implementation, and extension to MDP models.

\medskip
 \noindent
\textbf{Note:}
This pre-print is written in a tutorial style so it is accessible to new-comers.  It will be a part of a handout for upcoming short courses on RL.   A more compact version suitable for journal submission is in preparation.

\end{abstract} 
\thispagestyle{empty} 
%\setcounter{page}{0}
%

%%%%%%%%%%%%%%%%%%%%%%%%%%%%%%%%%%%%%%%%%%%%%%%%%%%%%%%%%%%%%%%%%%%%%%%%%%%%%%%%
%\begin{abstract} 
%  QAnon is a far-right conspiracy theory detailing a secret plot by the ``deep state''.
%This paper describes  a new approach to conspiracy theory, rooted in our own deep seated theoretical underpinnings.   
%
%\end{abstract}
%   \clearpage
%
% \tableofcontents
% 

  \clearpage

%%%%%%%%%%%%%%%%%%%%%%%%%%%%%%%%%%%%%%%%%%%%%%%%%%%%%%%%%%%%%%%%%%%%%%%%%%%%%%%%
\section{Introduction} 
\label{s:intro}

This paper concerns design of reinforcement learning algorithms for nonlinear, deterministic state space models.  The setting is primarily deterministic systems  in discrete time, where the main ideas are most easily described.   

Specifically,  we consider a state space model with state space $\state$,
and input (or action) space $\ustate$  (the sets $\state$ and $\ustate$ may be Euclidean space,   finite sets, or something more exotic).  
The input and state are related through the dynamical system   
\begin{equation}
x(k+1) = \Dds(x(k),u(k)) \,,\qquad k\ge 0\,, \  x(0)\in\state\,,
\label{e:dis_NLSSx_controlled_opt}
\end{equation}
where $\Dds\colon  \state\times\ustate \to\state $.   It is assumed that there is $(x^e,u^e)\in\state\times\ustate$ that achieves equilibrium:
\[
x^e = \Dds (x^e,u^e)
\]
The paper concerns infinite-horizon optimal control, whose definition requires a cost function $c\colon\state\to\Re_+$.
The cost function is  non-negative, and vanishes   at $(x^e,u^e)$.

%The total cost $J$ associated with a particular control input $\bfmu\eqdef \{u(k) : k\ge 0\}$ is defined by the sum
%\[
%J(\bfmu) = \sum_{k=0}^\infty  c (x(k), u(k))   
%\]

\mindex{Value function!Optimal}

The (optimal) value function is denoted
\begin{equation}
J^\oc (x) 
= \min_{\bfmu} \sum_{k=0}^\infty  c (x(k), u(k))    \, , \quad x(0)=x\in\state\, ,
\label{e:Jstar}
\end{equation}
where the minimum is over all input sequences $\bfmu\eqdef \{u(k) : k\ge 0\}$.
The goal of optimal control is to find an optimizing input sequence, and in the process we often need to compute the value function $J^\oc$.   
We settle for an approximation  in the majority of cases.   In this paper the approximations will be based on Q-learning; it is hoped that the main ideas will be useful in other formulations of reinforcement learning.

\paragraph{Background}    The dynamic programming equation associated with \eqref{e:Jstar} is
\begin{equation}
J^\oc (x) = \min_{u} \bigl\{ c (x,u)  +   J^\oc ( \Dds(x,u) ) \bigr\}  
\label{e:BE_part1}
\end{equation}
There may be state-dependent constraints, so the minimum is over a set $\ustate(x)\subset\ustate$.  
The function of two variables within the minimum in \eqref{e:BE_part1} is the ``Q-function'' of reinforcement learning:  
  \begin{equation}
Q^\oc(x,u) \eqdef  c (x,u)  +   J^\oc ( \Dds(x,u) )    
\label{e:Q_part1}
\end{equation}
so that the Bellman equation is equivalent to
\begin{equation}
J^\oc (x) = \min_{u} Q^\oc(x,u) 
\label{e:BE_part1_gen}
\end{equation}
From  \Cref{e:Q_part1,e:BE_part1_gen}  we obtain a fixed-point equation for the Q-function:
\begin{equation}
Q^\oc(x,u)  =  c (x,u) 
  +   \uQ^\oc ( \Dds(x,u) )
\label{e:WeAreQ}
\end{equation}
with $ \uQ(x) = \min_u Q(x,u)$ for any function $Q$.  The optimal input is state feedback $u^*(k) = \fee^\oc(x^\oc(k))$, with  
\choreP{  Is this too vague?  With this definition, $\fee^\oc $ might not be measurable.   How do you suggest we reword?}
\begin{equation}
\fee^\oc (x) \in \argmin_u  Q^\oc(x,u)\,,\qquad x\in\state
\label{e:fee}
\end{equation}

Temporal difference (TD)  and Q-learning are two large families of reinforcement learning algorithms based on approximating the value function or Q-function as a means to approximate  $\fee^\oc$ \cite{sutbar18,ber19a,bertsi96a}.   Consider a parameterized family $\{ Q^\theta : \theta\in\Re^d\}$;  each a  real-valued function on $\state\times \ustate$.
For example, the vector $\theta$ might represent weights in a neural network.  Q-learning algorithms are designed to approximate $Q^\oc$ within this parameterized family.   Given the parameter estimate $\theta\in\Re^d$, the ``$Q^\theta$-greedy policy'' is obtained:
\begin{equation}
\fee^\theta(x) = \argmin_u Q^\theta(x,u)
\label{e:Qgreedy}
\end{equation} 
Ideally, we would like to find an algorithm  that finds a value of $\theta$ so that this   best approximates the optimal policy.

Most algorithms are based on the sample path interpretation of \eqref{e:WeAreQ}:
\begin{equation}
Q^\oc(x(k),u(k) )  =  c (x(k),u(k) )   +   \uQ^\oc ( x(k+1))
\label{e:e:WeAreQ2}
\end{equation}
valid for any input-state sequence $\{ u(k), x(k) : k\ge 0\}$.     Two general approaches to define $\theta^\ocp$ are each posed in terms of the \textit{temporal difference}:
\begin{equation}
\Tdiff_{k+1}(\theta) \eqdef   -  Q^\theta (x(k),u(k) )   +   c(x(k),u(k))  +   \uQ^\theta (x(k+1)) 
\label{e:TDonline}
\end{equation} 
assumed to be observed on the finite time-horizon $0\le k\le\Nsam$.
\notes{which is where TD-learning gets its name
}
 
\begin{romannum}
\item  A  \textit{gold standard} loss function   is the mean-square Bellman error associated with \eqref{e:TDonline}: 
\begin{equation}
\clEBE(\theta)   =  \frac{1}{\Nsam} \sum_{k=0}^{\Nsam-1}  \bigl[   \Tdiff_{k+1}(\theta)  \bigr]^2
\label{e:MSBE}
 \end{equation}
 and  $\theta^\ocp$ is then defined to be its global minimum.
   Computation of a global minimum is a challenge since this function is not convex, even with    $\{ Q^\theta : \theta\in\Re^d\}$ linear in $\theta$.
\choreP{ $\clEBE$ is ugly.   It was previously  $\clE^e$, but the 'e' stands for equilibrium in this paper (and the book).  Any prettier symbols?}
 
 \item
 Watkins' Q-learning algorithm, as well as the earlier TD methods of Sutton (see \cite{sut84} for the early origins), 
 can be cast as a \textit{Galerkin relaxation} of the Bellman equation:   A sequence of $d$-dimensional \textit{eligibility vectors} $\{ \elig_k \} $ is constructed,  and the goal then is to solve 
\begin{equation}
\Zero  =  \frac{1}{\Nsam} \sum_{k=0}^{\Nsam-1}\Tdiff_{k+1}(\theta) \elig_k
\label{e:EmpiricalGalerkin}
\end{equation} 
\end{romannum}

The details of Watkins' algorithm can be found in the aforementioned references, along with \cite{devbusmey20} (which follows the notation and point of view of the present paper).   We note here only that in the original algorithm of Watkins,
\begin{romannum}
\item[$\grave$]   The algorithm is defined for finite state and action MDPs (Markov Decision Processes),   
and $Q^\oc \in \{ Q^\theta : \theta\in\Re^d\}$  (the goal is to compute the Q-function exactly).

\item[$\grave$]   The approximation family is linear $Q^\theta = \theta^\transpose \psi$,  with $\psi\colon\state\times\ustate\to\Re^d$.

\item[$\grave$]      $\elig_k = \psi(x(k),u(k))$. 
 \end{romannum}    
Equation \eqref{e:EmpiricalGalerkin} is not how Q-learning algorithms are typically presented, but does at least approximate the goal in many formulations.  A limit point of Watkins' algorithm, and generalizations such as  \cite{melmeyrib08,leehe19},  solves the ``projected Bellman equation'':  
\begin{equation}
\begin{aligned}
  \barf(\theta^\ocp) &= 0 \,,\qquad 
  \barf(\theta) =
\Expect
  \bigl[      \Tdiff_{k+1}(\theta) \elig_k   \bigr]   
\end{aligned} 
  \label{e:WatkinsRelax}
\end{equation}
where the expectation is in steady-state  (one assumption is the existence of a steady-state).
%said:  This can be regarded as a Galerkin relaxation of the Bellman equation.  
The   basic extension of Watkins' algorithm is defined by the recursion 
\begin{equation}
\theta_{n+1} = \theta_n + \alpha_{n+1}   \Tdiff_{n+1} \elig_n
\label{e:WatkinsFunctionApproximation}
\end{equation}
with $\{\alpha_n \}$ the non-negative step-size sequence,   and $  \Tdiff_{n+1}$ is short-hand for $  \Tdiff_{n+1}(\theta_n) $.    
The recursion \eqref{e:WatkinsFunctionApproximation} is called Q(0)-learning for the special case   $\elig_n =  \nabla_\theta Q^{\theta} (x(n),u(n) ) \big|_{\theta =\theta_n}$, in analogy with TD(0)-learning~\cite{devbusmey20}. 
Criteria for convergence is typically cast within the theory of stochastic approximation, which is based on the ODE  (ordinary differential equation),
\begin{equation}
\ddt \odestate_t = \barf(\odestate_t)   \,,\qquad \odestate_0\in\Re^d
\label{e:WatkinsRelaxODE}
\end{equation}
 Conditions for convergence of \eqref{e:WatkinsFunctionApproximation} or its ODE approximation \eqref{e:WatkinsRelaxODE}  are very restrictive  \cite{melmeyrib08,leehe19}.

%In analogy with these origins, it is typical to define  $\elig_k =  \nabla_\theta Q^{\theta} (x(k),u(k) ) \Big|_{\theta =\theta_k}$ in extensions of Watkins Q-learning algorithm.

While not obvious from its description, the   DQN algorithm (a significant component of famous applications such as AlphaGo) will converge to the same projected Bellman equation, provided it is convergent---see \Cref{t:dumbDQN} below.   

This opens an obvious question:  does \eqref{e:WatkinsRelax} have a solution?  Does the solution lead to a good policy?   An answer to the second question is wide open, despite the success in applications.  The answer to the first question is, in general, \textit{no}.     The conditions imposed in  \cite{melmeyrib08} for a solution are very strong, and not easily verified in any applications;  the more recent work \cite{leehe19} offers improvements, but nothing approaching a full understanding of the algorithm.
  
The question of existence led to an entirely new approach in  \cite{maeszebhasut10}, based on 
    the  non-convex optimization problem:
\choreS{
Moreover, there are counterexamples:    I'm working on refs on this,   
\\
\rd{see final NeurIPS paper:}
}
\begin{equation}
\min_\theta J(\theta)=\min_{\theta}\half \barf(\theta)^\intercal M\barf(\theta),\qquad \mbox{with} \quad M>0
\label{e:gq-obj}
\end{equation}
with $\barf$ defined in \eqref{e:WatkinsRelax}.
Consider the continuous-time gradient descent algorithm associated with  \eqref{e:gq-obj}:
\begin{equation}
\ddt \odestate_t =  -   [\partial_\theta  \barf\, (\odestate_t)]^\transpose M\barf(\odestate_t)
\label{e:GQODE}
\end{equation} 
The GQ-learning algorithm is a stochastic approximation (SA) translation of this ODE,  using $M = \Expect[\elig_n\elig_n^\transpose]^{-1}$.   
Convergence holds under conditions, such as a coercive condition on $J$ (which   is not easily verified a-priori).  Most important:
does this algorithm lead to a good approximation of the Q-function, or the policy?  

This brings us back to (ii):  \textit{why   approximate the solution of \eqref{e:WatkinsRelax}?}
% with  $\elig_k  =   \nabla_\theta Q^{\theta_k} (x(k),u(k) )$?}
We do not have an answer.  In this paper we return to \eqref{e:MSBE}, for which we seek convex re-formulations.

  \paragraph{Contributions.}  

The apparent challenge with approximating the Q-function is that it solves a nonlinear fixed point equation \eqref{e:WeAreQ} or \eqref{e:e:WeAreQ2},  for which root finding problems for approximation may not be successful  (as counter examples show).   This challenge seems paradoxical, given the long history of convex analytic approaches to dynamic programming in both the MDP literature \cite{man60a,der70,bor02a}    and the linear optimal control literature  \cite{vanboy99}.     

\saveForBook{
Sean to revisit Daniela de Farias' work  
\cite{farroy03a,farroy06}  (nearly done)}

The starting point of this paper is to clarify this paradox,  and from this create a new family of RL algorithms designed to minimize a convex variant of the empirical mean-square error \eqref{e:MSBE}.  This step was anticipated in   \cite{mehmey09a}, for which a convex formulation of Q-learning was proposed for deterministic systems in continuous time;   in this paper we   call this \textit{H-learning},  motivated by commentary in the conclusions of   \cite{mehmey09a}.   The most important ingredient in H-learning is the creation of a  convex program suitable for RL based on an over-parameterization, in which the value function and Q-function are treated as separate variables.

A significant failing of H-learning was the complexity of implementation in its ODE form, because of implicit constraints on $\theta$:    complexity persists when using a standard Euler approximation to obtain a recursive algorithm.  In 2009, the authors believed that recursive algorithms were essential for the sake of data efficiency.
Motivation for the present work came from breakthroughs in empirical risk minimization (ERM) appearing in the decade since \cite{mehmey09a} appeared,   and the success of DQN algorithms that are based in part on ERM.

The main contributions are summarized as follows:
\begin{romannum}
 
\item   
The linear programming (LP) approach to dynamic programming has an ancient history \cite{man60a,der70,bor02a},  with applications to approximate dynamic programming beginning with de Farias' thesis \cite{farroy03a,farroy06}.   The   ``\DPLP''  (dynamic programming linear program) is introduced in \Cref{t:LPforJQ} for deterministic control systems,  with generalizations to MDPs in \Cref{s:H-MDP}.   
  \textit{It is the over-parameterization that lends itself to data driven RL algorithms for nonlinear control systems.}
The relationship with semi-definite programs for LQR is made precise in \Cref{t:LPforLQR}.

\item   
\Cref{s:Q} introduces   new Q-learning algorithms inspired by the \DPLP.     Our current favorite is Batch Convex Q-Learning \eqref{e:ConvexQalmostSA}:  it appears to offer the most flexibility in terms of computational complexity in the optimization stage of the algorithm.    A casual reader might mistake \eqref{e:ConvexQalmostSA} for the DQN algorithm.   The algorithms are very different, as made clear in  \Cref{t:dumbDQN,t:ConvexDQNconverges}:    the DQN algorithm cannot solve the minimal mean-square Bellman error optimization problem.   Rather, any limit of the algorithm must solve the   relaxation \eqref{e:WatkinsRelax}  (recall that there is little theory providing sufficient conditions for a solution to this fixed point equation).

The algorithms proposed here converge to entirely different values:

 \item    As in \cite{mehmey09a}, it is argued that there is no reason to introduce random noise for exploration for RL in deterministic control applications.    Rather, we opt for the quasi-stochastic approximation (QSA) approach of \cite{mehmey09a,shimey11,berchecoldalmehmey19b,cheberdevmey19}.     Under mild conditions, including a linear function approximation architecture,  parameter estimates from the Batch Convex Q algorithm will converge to the solution to the quadratic program  
\begin{equation}
\begin{aligned}
\min_\theta \ & 
-
    \langle \mu,J^\theta \rangle    
+
\kappaBE
 \Expect
  \bigl[    \bigl\{  \Tdiff^\circ_{k+1}(\theta)   \bigr\}^2   \bigr]   
   \\
\st  \ &  \Expect
  \bigl[    \Tdiff^\circ_{k+1}(\theta)  \elig_k(i)    \bigr]       \ge 0  \,,\qquad i\ge 1
\end{aligned} 
\label{e:MMQ}
\end{equation}
where 
 $ \Tdiff^\circ_{k+1}(\theta)  $ denotes the \textit{observed Bellman error} at time $k+1$:
\[
\Tdiff^\circ_{k+1}(\theta) \eqdef   -Q^\theta (x(k),u(k) )   +   c(x(k),u(k))  +   J^\theta (x(k+1)) 
\]
The expectations in \eqref{e:MMQ} are in steady-state. 
The remaining variables are part of the algorithm design:
 $\{ \elig_k  \}$ is a non-negative vector-valued sequence,   $\mu$ is a positive measure, and      $\kappaBE$ is non-negative. 

The quadratic program \eqref{e:MMQ}
 is a relaxation of the \DPLP, which is tight under ideal conditions (including the assumption that   $(J^\oc,Q^\oc)$ is contained in the function class).   See   \Cref{t:ConvexQ_limit_cor} for details.

\end{romannum}

\paragraph{Literature review}

This paper began as a resurrection of the conference paper  \cite{mehmey09a}, which deals with Q-learning for deterministic systems in continuous time.   A primitive version of convex Q-learning was introduced; the challenge at the time was to find ways to create a reliable  online algorithm.  This is resolved in the present paper through a combination of Galerkin relaxation techniques and ERM.
Complementary to the present research is the empirical    value iteration algorithms developed in
\cite{gupjaigly15,shajaigup19}.

The ideas in  \cite{mehmey09a} and the present paper were   inspired in part by the LP approach to dynamic programming introduced by Manne in the 60's \cite{man60a,der70,araborferghomar93,bor02a}.    A significant program on linear programming approaches to approximate dynamic programming is presented in  \cite{farroy03a,defvan04,farroy06}, but we do not see much overlap with the present work.   
There is a also an on-going research program on LP approaches to optimal control for deterministic systems  \cite{vin93,herhertak96,herlas02a,las10,kamsutmohlyg17,borgai19,gaiparshv17,gaiqui13,borgaishv19},  and  semi-definite programs (SDPs) in linear optimal control \cite{boyelgferbal94,wanboy09}. 
\saveForBook{
\cite{fingaileb08,,}   
\\
what is this?  Primal-dual Q-learning framework for LQR design}

Major success stories in Deep Q-Learning (DQN) practice are described in \cite{DQN15,mnikavsilgraantwierie13,DQN16Asynchronous},  and   recent surveys \cite{ansbarshi17,shaman20}.
We do not know if the research community is aware that these algorithms, if convergent, will converge to the projected Bellman equation \eqref{e:WatkinsRelax}, and that an ODE approximation of DQN is identical to that of \eqref{e:WatkinsFunctionApproximation}  (for which stability theory is currently very weak).   A formal definition of ``ODE approximation'' is provided at the end of \Cref{s:QSA}.

Kernel methods in RL have a significant history, with most of the algorithms designed to approximate the value function for a fixed policy,
so that the function approximation problem can be cast in a least-squares setting \cite{glyorm02b,fenliliu19}  (the latter proposes extensions   to Q-learning).

\saveForBook{
   Key background is \cite{rie05}, and explosions from Minh et al }

\choreP{Essay on your fav contributions}

\choreS{ mention of Ben Recht's work}

\saveForBook{
The LP \eqref{e:JQz}   is surely well known.   Also,
is the LQR corollary \Cref{t:LPforLQR} equivalent to the standard representation from textbooks?
}

\saveForBook{ Nice survey material in \cite{wanboy09}.  References to check out include ... see commented text}
 
% S. Boyd, L. El Ghaoui, E. Feron, V. Balakrishnan, Linear Matrix Inequalities in
%Systems and Control Theory, SIAM Books, Philadelphia, 1994.
% 
% with variable P. The optimal point is P = P#, and the optimal value
%of this problem is J# [8,1,20,3,7,29].
% 
% [1] M. Ait-Rami, L. El Ghaoui, Solving non-standard Riccati equations using LMI
%optimization, in: Proceedings of the 34th Conference on Decision and Control, 1995
%. 
%[3] F.A. Badawi, On a quadratic matrix inequality and the corresponding algebraic
%Riccati equation, International Journal of Control 36 (1982) 313–322.
% 
%[7] S. Bittanti, A.J. Laub, J.C. Willems, The Riccati Equation, Springer-Verlag, 1991.
% 
%[8] S. Boyd, L. El Ghaoui, E. Feron, V. Balakrishnan, Linear Matrix Inequalities

\notes{LSVI etc seems like bullshit -- see commented text}
%\choreS{Lots of work here on LSVI, etc}
%\bl{
% \rd{need to see connections:} \textit{Least-squares value iteration} (LSVI),
%and Ben's
%randomized least-squares value iteration (RLSVI) \cite{osbvanwen16},  and more language!  \textit{episodic reinforcement learning} \cite{danbru15}.   References \cite{beriof96} \cite{lagparlit02}  for background on LSVI.  
%\\
%See thesis of \cite{hua18}, and background:
%LSPI: 
%Lagoudakis and Parr, 2001
%\\
%Least-Squares
%Q-Learning:
%[Lagoudakis and Littman, 2000]  
%}

\saveForBook{
  \textbf{Empirical risk} 
Bach and Moulines have convinced the world to  use batch methods instead of stochastic approximation in some settings   (I know these papers shifted the field, but I need to research: \cite{moubac11,bacmou13}).   Message to Eric on June 29: \textit{
Many cite your 2011 paper as justification (2 years back, Giannakis told me that SA is now obsolete, and used your work as evidence!  I’m sure he was exaggerating). I was blown away at your Data Science Summer School when I learned that all of the stochastic optimization crowd had sworn off SA.  I understand now: they basically get optimal variance without second order methods.  Perhaps you can provide solid references from your early days. 
}
Peter RICHTÁRIK was key player at the summer school.
}

   \smallbreak

The remainder of the paper is organized as follows.   \Cref{s:OPTSA} begins with a review of linear programs for optimal control,  and new formulations designed for application in RL.   
Application to LQR is briefly described in \Cref{s:LQR}, which concludes with an example to show the challenges expected when using standard Q-learning algorithms even in the simplest scalar LQR model.   \Cref{s:Q} presents a menu of ``convex Q-learning" algorithms.
Their relationship with DQN (as well as significant advantages) is presented in \Cref{s:DQN}.
Various extensions and refinements are collected together in \Cref{s:junkyard}, including application to MDPs.
   The theory is illustrated with
   a single example in
   \Cref{s:ex}.
\Cref{s:conc} contains conclusions and topics for future research.
Proofs of technical results are contained in the appendices.

\section{Convex Q-Learning}
\label{s:OPTSA}

The RL algorithms introduced in \Cref{s:Q} are designed to approximate the value function $J^\oc$ within a finite-dimensional function class $\{J^\theta :\theta\in\Re^d\}$.   We obtain a convex program when this parameterization is linear.   For readers interested in discounted cost, or finite-horizon problems,   \Cref{s:var} contains hints on how the theory and algorithms can be modified to your favorite optimization criterion.

As surveyed in the introduction, a favored approach in Q-learning is to consider  for each $\theta\in\Re^d$  the associated Bellman error 
\begin{equation}
\BE^\theta(x) =   -J^\theta (x) + \min_{u}  [ c (x,u)  +   J ^\theta( \Dds(x,u) )  ]  \,,\qquad x\in\state
\label{e:BErr_theta_part1}
\end{equation}
A generalization of \eqref{e:MSBE} is to introduce  a weighting measure $\mu$  (typically a probability mass function (pmf) on $\state $),  and consider the mean-square Bellman error
\saveForBook{Measure ok in book?}
\begin{equation}
\clE(\theta) =  \int  [\BE^\theta(x) ]^2 \, \mu(dx)
\label{e:BErr_MSE_part1}
\end{equation}
A significant challenge is that  the loss function $\clE$ is not convex.   
We obtain a convex optimization problem that is suitable for application in RL by applying a common trick in optimization:  over-parameterize the search space.    Rather than approximate $J^\oc$,   we simultaneously approximate $J^\oc$ and $Q^\oc$,   where the latter is the Q-function defined in  \eqref{e:Q_part1}.

Proofs of all technical results in this section can be found in the appendices.

\subsection{Bellman Equation is a Linear Program}
\label{s:BELP}

For any function $J\colon\state\to\Re$, and any scalar $r$,  let $S_J(r)$ denote the \textit{sub-level set}:
\begin{equation}
S_J(r) =  \{ x \in\state : J(x) \le r\}
\label{e:SJ}
\end{equation}
The function $J$ is called \textit{inf-compact} if   the set $S_J(r)$ is either pre-compact, empty,  or $S_J(r) =\state$  (the three possibilities depend on the value of $r$).   In most cases we find that $S_J(r) =\state$ is impossible, so that we arrive at the  stronger
  \textit{coercive} condition:
\mindex{Coercive}
\begin{equation}
\lim_{\|x\|\to\infty} J(x) =\infty 
\label{e:coercive}
\end{equation}

\begin{proposition}
\label[proposition]{t:LPforJQ}
Suppose that the value function $J^\oc$ defined in \eqref{e:Jstar} is continuous, inf-compact, and vanishes only at $x^e$.   Then,  the pair
$(J^\oc,Q^\oc)$  solve  the following convex program in the ``variables'' $(J,Q)$:  
 \begin{subequations}
\begin{align} 
\max_{J,Q}   \ \  &   \langle \mu, J \rangle      
\label{e:ConvexBEa}
\\
\st  \ \  &  Q(x,u) \le  c (x,u)  +   J ( \Dds(x,u) )  
\label{e:ConvexBEb}
\\
&   Q(x,u) \ge J(x)  \,,\qquad\qquad  x\in\state\,,  \  u\in\ustate(x)
\label{e:ConvexBEc}
\\[.5em]
&  \text{$J$ is continuous, and $ J(x^e) =0$. }
\label{e:ConvexBEd}
\end{align} 
\label{e:ConvexBE}%
\end{subequations}%
\end{proposition}%

We can without loss of generality strengthen \eqref{e:ConvexBEb} to equality:   $Q(x,u)  = c (x,u)  +   J ( \Dds(x,u) )$.
Based on this substitution, the variable $Q$ is eliminated:
 \begin{subequations}
\begin{align} 
\max_J   \ \  &   \langle \mu, J \rangle      
\label{e:ConvexBEJa}
\\
\st  \ \  & c (x,u)  +   J ( \Dds(x,u) )  \ge J(x)    \,,\qquad\qquad  x\in\state\,,  \  u\in\ustate(x)
\label{e:ConvexBEJc}
\end{align} 
\label{e:ConvexBEJ}%
\end{subequations}%
This more closely resembles what you find in the MDP literature (see \cite{araborferghomar93} for a survey).   

The more complex LP \eqref{e:ConvexBE} is introduced because it is   easily adapted to RL applications.  To see why,  define for any  $(J,Q)$ the Bellman error:
\begin{equation}
\Tdiff^\circ(J,Q)_{(x,u)}  \eqdef   -  Q (x,u )   +   c (x,u)  +   J (\Dds(x,u) ) \,,  \qquad x\in\state\,\ u\in\ustate 
\label{e:TD_JQ}
\end{equation} 
Similar to the term $ \Tdiff_{n+1}$   appearing in \eqref{e:WatkinsFunctionApproximation}, we have   for any input-state sequence,
\[
\Tdiff^\circ(J,Q)_{(x(k),u(k))}  =   -  Q (x(k), u(k))   +   cx(k), u(k))  +   J ( x(k+1) )
\]
Hence we can observe the Bellman error along the sample path.   The right hand side will be called the temporal difference,  generalizing the standard terminology.

We present next an important corollary that will motivate RL algorithms to come.    Adopting the notation of \Cref{e:MSBE,e:MMQ}, denote  $\clEBE(J,Q) = \Tdiff^\circ(J,Q)^2$:  this is  a non-negative  function on $\state\times\ustate$, for any pair of functions $J,Q$.

\begin{corollary}
\label[corollary]{t:LPforJQ_cor}
Suppose that the assumptions of \Cref{t:LPforJQ} hold.    Then,  for any constants $\kappaBE>0$, $0\le \fudgeBE \le 1$,  and pmfs $\mu,\nu$,
the pair
$(J^\oc,Q^\oc)$  solve  the following quadratic program:  
 \begin{subequations}
\begin{align} 
\max_{J,Q}   \ \  &    \langle \mu, J \rangle  -  \kappaBE \langle \nu, \clEBE(J,Q)  \rangle            
   \\
  \st  \ \           & \textit{Constraints \eqref{e:ConvexBEb}--\eqref{e:ConvexBEd}  }    
  \\
  &  Q(x,u) \ge  (1-\fudgeBE) c (x,u)  +   J ( \Dds(x,u) )  
  \label{e:QuadBEa_tightenBE} 
\end{align} 
\label{e:QuadBEa}%
\end{subequations}%
\end{corollary}%

The extra constraint   \eqref{e:QuadBEa_tightenBE}  is introduced so we arrive at something closer to \eqref{e:ConvexBEJ}.     The choice $\fudgeBE=0$ is not excluded, but we will need the extra flexibility when we seek approximate solutions.

%
%\begin{corollary}
%\label[corollary]{t:LPforJQ_cor}
%Suppose that the assumptions of \Cref{t:LPforJQ} hold.    Then,  for any constant $\kappaBE>0$, and positive measures $\mu,\nu$,
%the pair
%$(J^\oc,Q^\oc)$  solve  the following quadratic program: 
%\saveForBook{pmf or positive measure?}
%\begin{equation}
%\begin{aligned}
%\max_{J,Q}   \ \  &    \langle \mu, J \rangle  -  \kappaBE \langle \nu, \clEBE(J,Q)  \rangle            
%   \\
%  \st  \ \           & \textit{Constraints \eqref{e:ConvexBEb}--\eqref{e:ConvexBEd}  }    
%\end{aligned} 
%\label{e:QuadBEa}
%\end{equation}
%\end{corollary}%

\subsection{Semi-definite program for LQR}
\label{s:LQR}

Consider the LTI model:
\begin{subequations}
\begin{align}
x(k+1) &=  F x(k) + G u(k), \qquad x (0) = x_0
\label{e:dis_LSSx_controlled}%
\\
y(k) & = H x(k)  
\end{align}
\label{e:dis_LSS_controlled}%
\end{subequations}%
where $(F,G,H)$ are matrices of suitable dimension (in particular, $F$ is $n\times n$ for an $n$-dimensional state space), and assume that the cost is quadratic:
\begin{equation}
c(x,u) = \| y\|^2 + u^\transpose R u 
\label{e:quadraticCost}
\end{equation}
with $y = H x$,   and $R>0$.   We henceforth denote  $S = H^\transpose H \ge 0$.  

To analyze the LP \eqref{e:ConvexBE}  for this special case,  it is most convenient to express all three functions appearing in  \eqref{e:ConvexBEb}  in terms of the   variable $z^\transpose  = (x^\transpose,  u^\transpose)$:    
\begin{subequations}
\begin{alignat}{3}
J^\oc(x,u)  &=z^\transpose {M^J}^\oc z  
				&	Q^\oc(x,u)   =z^\transpose {M^Q}^\oc z  \qquad c(x,u) &= z^\transpose M^c z 
\\[.5em]
 {M^J}^\oc&= \begin{bmatrix}  M^\oc & \Zero \\ \Zero & \Zero\end{bmatrix} 
 &
	 M^c& =   \begin{bmatrix}  S & \Zero \\ \Zero &R \end{bmatrix} 
\\
& &{M^Q}^\oc  = M^c + 
	\begin{bmatrix}
		 F^TM^\oc F & F^TM^\oc G \\
		G^TM^\oc F &  G^TM^\oc G \\
	 \end{bmatrix}   &
	 \label{e:JQzQ}
\end{alignat} 
\label{e:JQz}%
\end{subequations}  
Justification of the formula for  ${M^Q}^\oc$ is contained in the proof of   \Cref{t:LPforLQR} that follows.

\begin{proposition}
\label[proposition]{t:LPforLQR}
Suppose that $(F,G)$  is stabilizable and  $(F,H)$  is detectable,   so that  $J^\oc $ is everywhere finite.   
%%% For book:
%%%%Suppose that the value function   $J^\oc $ is everywhere finite for the Linear Quadratic Regulator problem.   
Then, the value function   and   Q-function     are each quadratic:     $J^\oc (x) = x^\transpose M^\oc x$ for each $x$,   where $ M^\oc \ge 0$ is a solution to the algebraic Riccati equation,  and the quadratic Q-function is given in \eqref{e:JQz}.
The matrix $M ^\oc$ is also the solution to the following convex program:
 \begin{subequations}
\begin{align} 
M ^\oc \in 
\argmax   \ \  &    \trace(M)    &&
\label{e:ConvexLQRa}
\\
\st  \ \  &   \begin{bmatrix}  S & \Zero \\ \Zero &R \end{bmatrix} 
		    + 
			\begin{bmatrix} 
			F^TMF & F^TMG
				 \\ 
			G^TMF & G^TMG
			\end{bmatrix}  
			\ge
			  \begin{bmatrix} 
			M &  0 
				 \\ 
			0 & 0
			\end{bmatrix}  &&  
\label{e:ConvexLQRc}
\end{align} 
\label{e:ConvexLQR}%
\end{subequations} 
where  the maximum is   over symmetric matrices $M$, and
 the inequality constraint    \eqref{e:ConvexLQRc}
  is in the sense of symmetric matrices. \end{proposition}

 Despite its linear programming origins,   \eqref{e:ConvexLQR} is not a linear program:    it  is an example of a semi-definite program (SDP),  for which there is a rich literature with many applications to control  \cite{vanboy99}.

\paragraph{LQR and DQN}

The class of LQR optimal control problems is a great vehicle for illustrating the potential challenge with popular RL algorithms, and how these challenges might be resolved using the preceding optimal control theory.

Recall that both DQN and  Q(0)-learning, when convergent,  solves the root finding problem \eqref{e:WatkinsFunctionApproximation}.     We proceed by identifying the vector field $\barf$ defined in   \eqref{e:WatkinsRelax} (and appearing in  \eqref{e:WatkinsFunctionApproximation}) for the LQR problem,   subject to the most natural assumptions on the function approximation architecture.

For the linear system \eqref{e:dis_LSS_controlled} with quadratic cost \eqref{e:quadraticCost}, the Q-function is a quadratic (obtained from the algebraic Riccati equation (ARE) or  the SDP  \eqref{e:ConvexLQR}).   To apply Q-learning we might formulate a linear parameterization:
\[
Q^\theta(x,u ) = x^\transpose M_F x + 2u^\transpose N x + u^\transpose M_G u
\]
in which the three matrices depend linearly on $\theta$.      The Q-learning algorithm  \eqref{e:WatkinsFunctionApproximation} and  the DQN algorithm require  the minimum of the Q-function, which is easily obtained in this case:
\[
\uQ^\theta(x) = \min_u Q^\theta(x,u ) = x^\transpose  \bigl \{ M_F  - N^\transpose M_G^{-1} N \bigr\} x
\]
Consequently,  the  function $\uQ^\theta$ is typically a highly nonlinear function of $\theta$, and is unlikely to be Lipschitz continuous.
It follows that the same is true for the vector field $\barf$ defined in   \eqref{e:WatkinsRelax}.

%along with the minimizing value   $u^\oc = \fee^\theta(x) = -M_G^{-1} Nx $.    

The challenges are clear even for the \textit{scalar} state space model,  for which we write
\[
Q^\theta(x,u ) =\theta_1  x^2 + 2 \theta_2 xu  + \theta_3 u^2\,,\qquad \uQ^\theta (x) =  \bigl \{ \theta_1  - \theta_2^2/\theta_3 \bigr\} x^2
\]
The vector field    \eqref{e:WatkinsRelax}    becomes, 
\begin{gather*}
\barf(\theta) = -  \Sigma_\elig \theta + b_c +   \bigl \{ \theta_1  - \theta_2^2/\theta_3 \bigr\}b_x 
\\
 b_c = 
\Expect_\varpi[ c(x(k),u(k))  \elig_k  ]
  \,,\qquad
 b_x =  \Expect_\varpi[  x(k+1)^2  \elig_k  ] 
  \,,\qquad
 \Sigma_\elig = \Expect_\varpi[    \elig_k {\elig_k}^\transpose  ] 
\end{gather*}
where the expectations are in ``steady-state'' (this is made precise in \eqref{e:QSAergodic} below).   The vector field is far from Lipschitz continuous, which rules out stability analysis through stochastic approximation techniques without projection onto a compact convex region that excludes  $\theta_3=0$.    

Projection is not a problem for this one-dimensional special case.  
It is not known how to stabilize these Q-learning algorithms or their ODE approximation \eqref{e:WatkinsRelaxODE}
  if the state space has dimension greater than one.  We also do not know if the roots of $\barf(\theta)$ provide good approximations of the ARE, especially   if $Q^\oc$ does not lie in the function approximation class.

\subsection{Exploration and Quasi-Stochastic Approximation}
\label{s:QSA}

The success of an RL algorithm based on temporal difference methods depends on the choice of input $\bfmu$ during training.   The purpose of this section is to make this precise, and present our main assumption on the input designed for generating data to train the algorithm (also known as \textit{exploration}).

\notes{Said!
\\
To construct an RL algorithm to approximate any of the convex programs surveyed in \Cref{s:BELP}, 
\saveForBook{chapter in book!}
  we need to construct an input sequence for exploration.   
  }
  
Throughout the remainder of the paper it is assumed that the input used for training is state-feedback with perturbation, of the form
\begin{equation}
u(k) = \fee(x(k),  \qsaprobe(k) ) 
\label{e:randomizedStateFB}
\end{equation}  
where the exploration signal $\bfqsaprobe$ is a bounded sequence evolving on  $\Re^p$ for some $p\ge 1$.
We adopt the quasi-stochastic approximation (QSA) setting of \cite{mehmey09a,shimey11,berchecoldalmehmey19b,cheberdevmey19}.  For example, $\qsaprobe(k)$ may be a mixture of sinusoids of irrational frequencies.   It is argued in \cite{shimey11,berchecoldalmehmey19b} that this can lead to substantially faster convergence in RL algorithms, as compared to the use of an i.i.d.\ signal.   

It will be convenient to assume that the exploration is Markovian, which for a deterministic sequence means it can be expressed
\begin{equation}
\qsaprobe(k+1) = \qsaDyn(\qsaprobe(k) )
\label{e:probeMarkov}
\end{equation}
It is assumed that   $\qsaDyn \colon\Re^p\to\Re^p$ is continuous.
Subject to the policy \eqref{e:randomizedStateFB}, it follows that the triple $\Phi(k) = (x(k),u(k),\qsaprobe(k))^\transpose$ is also Markovian,  with state space $\zstate = \state\times\ustate\times \Re^p$. 
The continuity assumption imposed below ensures that $\bfPhi$ has the Feller property,  and boundedness of $\bfPhi$ from just one initial condition implies the existence of an invariant probability measure that defines the ``steady state'' behavior \cite[Theorem~12.0.1]{MT}.

For any continuous function $g\colon\zstate\to\Re$ and $\Nsam\ge 1$,  denote  
\[
\barg_\Nsam = \frac{1}{\Nsam}  \sum_{k=1}^\Nsam  g(\Phi(k)) 
\]
This is a deterministic function of the initial condition $\Phi(0)$. 
For any $L>0$ denote 
\[
\clG_L = \{ g :   \|  g(z') - g(z)  \| \le L \| z-z'\|,\   \text{for all }\ z,z'\in\zstate \}
\]
The following is assumed throughout this section.  
\begin{quote}
\textbf{(A$\bfqsaprobe$)} \
The state and action spaces $\state$ and $\ustate$ are Polish spaces;
 $\Dds$ defined in \eqref{e:dis_NLSSx_controlled_opt},  $\fee$ defined in \eqref{e:randomizedStateFB},   and $\qsaDyn$ in \eqref{e:probeMarkov}
are each continuous on their domains,  and the larger state process $\bfPhi$ is  bounded and   ergodic in the following sense:  
There is a unique probability measure $\varpi$, with compact support, such that  for any continuous function $g\colon\zstate\to\Re$, the following ergodic average exists for each initial condition $\Phi(0)$
\begin{equation}
\Expect_\varpi[ g(\Phi)] \eqdef \lim_{\Nsam\to\infty} \barg_\Nsam
\label{e:QSAergodic}
\end{equation} 
Moreover, the limit is uniform on $\clG_L$,  for each  $L<\infty$,
\[
\lim_{\Nsam\to\infty}   \sup_{g\in\clG_L} | \barg_N - \Expect_\varpi[ g(\Phi)]  | = 0
\]
\end{quote}
For analysis of Q-learning with function approximation, the vector field   $\barf$ introduced in   \eqref{e:WatkinsRelax}  is defined similarly:
 \[
	  \barf(\theta) = \lim_{N\to\infty}\frac{1}{N} \sum_{k=1}^{N}     \Bigl[ -  Q^\theta (x(k),u(k) )   +   c(x(k),u(k))  +   \uQ^\theta (x(k+1)) \Bigr] \elig_k  
\]

Assumption~(A$\bfqsaprobe$) is surely \textit{far stronger than required}.   If $\bfPhi$ is a Feller Markov chain (albeit deterministic), it is possible to consider sub-sequential limits defined by one of many possible invariant measures.  Uniqueness of $\varpi$ is assumed so we can simplify the description of   limits of the algorithm.    The uniformity assumption  is used in   Appendix~\ref{app:CQL}
 to get a simple proof of ODE approximations,  starting with a proof that the algorithm is stable in the sense that the iterates are bounded.   In Borkar's second edition \cite{bor20a} he argues that such strong assumptions are not required to obtain algorithm stability, or ODE approximations.      \notes{Firm this up some day}

The Kronecker–Weyl Equidistribution Theorem provides ample examples of signals satisfying the   ergodic limit \eqref{e:QSAergodic}  when $g$ depends only on the exploration signal   \cite{bec17,kuinie12}.   Once this is verified, then the full Assumption~(A$\bfqsaprobe$) will hold is $\bfPhi$ is an e-chain (an equicontinuity assumption introduced by Jamison) \cite{MT}.

The choice of ``exploration policy'' \eqref{e:randomizedStateFB} is imposed mainly to simplify analysis.  We might speed convergence significantly with an ``epsilon-greedy policy'':
\[
u(k) = \fee^{\theta_k}(x(k),  \qsaprobe(k) ) 
\]
in which the right hand side is a perturbation of the exact $Q^\theta$-greedy policy \eqref{e:Qgreedy}, using current estimate $\theta_k$.
There is a growing literature on much better exploration schemes, motivated by techniques in the bandits literature \cite{osbroydanzhe17,osbvanwen16}.    The marriage of these techniques with the algorithms introduced in this paper is a subject for future research.

\paragraph{ODE approximations}

The technical results that follow require that we make precise what we mean by  \textit{ODE approximation} for a recursive algorithm.    Consider a recursion of the form  
\begin{equation}
\theta_{n+1}  = \theta_n  +    \alpha_{n+1}   f_{n+1}(\theta_n)\,,\qquad n\ge 0
\label{e:BCQLSAa_tutorial}
\end{equation}
in which  $\{f_n\}$ is a sequence of functions that admits an ergodic limit:
\[
\barf(\theta) \eqdef \lim_{N\to\infty}\frac{1}{N} \sum_{k=1}^N f_k(\theta)  \,,\qquad \theta\in\Re^d
\]
The associated ODE is defined using this vector field:
\begin{equation}
\ddt \odestate_t= \barf(\odestate_t) 
\label{e:ODEideal}
\end{equation}

An ODE approximation is defined by mimicking the usual Euler construction:  the time-scale for the ODE is defined by the non-decreasing time points $t_0=0$ and $t_n = \sum_0^n \alpha_k$ for $n\ge 1$.    Define a continuous time process $\bfODEstate$ by  $\ODEstate_{t_n}= \theta_n$ for each $n$,  and extend to all $t$ through piecewise linear interpolation.    The next step is to fix a time horizon for analysis of length $\Hor>0$,  where the choice of $\Hor$ is determined based on properties of the ODE.       Denote $\Hor_0=0$, and
\begin{equation}
\Hor_{n+1} = \min\{ t_n :   t_n - \Hor_n \ge \Hor \}  \,,\qquad n\ge 0
\label{e:ODE_TimeInterval}
\end{equation}
Let $\{ \odestate_t^n : t\ge \Hor_n\}$ denote the solution to the ODE  \eqref{e:ODEideal} with initial condition  
$ \odestate^n_{\Hor_n}  =  \theta_{k(n)}$,   with index defined so that $t_{k(n)} = \Hor_n$.
 We then say that the algorithm \eqref{e:BCQLSAa_tutorial} admits an \textit{ODE approximation} if for each initial $\theta_0$,
\begin{equation}
 \lim_{n\to\infty} \sup_{\Hor_n\le t\le \Hor_{n+1}} \| \ODEstate_t  - \odestate_t^n \|  = 0
\label{e:ODE_Approximation_Defined}
\end{equation}

%%%%%%%%%%%%%%%%%%%%%%%%%%%%%%%%%%%%%%%%%%%%%%%%%%%%%%%%%%%%%%%%%%%%%%%%%%%%%%%%
\subsection{Convex Q-learning} 
\label{s:Q}

The  RL algorithms introduced in this paper are all motivated by the ``\DPLP'' \eqref{e:ConvexBE}.   We search for an approximate solution among a finite-dimensional family 
 $\{ J^\theta\,,\ Q^\theta : \theta\in\Re^d\}$.  The value $\theta_i$ might represent the $i$th weight in a neural network function approximation architecture,  but to justify the adjective \textit{convex} we require a linearly parameterized family:
\begin{equation}
 J^\theta(x) = \theta^\transpose \psi^J (x)\,,\qquad 
 Q^\theta(x,u) = \theta^\transpose \psi (x,u) 
\label{e:linParFamily}
\end{equation}
The function class is normalized with $ J^\theta(x^e) =0$ for each $\theta$.  For the  linear approximation architecture this requires $ \psi^J_i (x^e) =0$ for each $1\le i\le d$; for a neural network architecture, this normalization is imposed through definition of the output of the network.  
Convex Q-learning based on a \textit{reproducing kernel Hilbert space} (RKHS) are contained in \Cref{s:kernel}.

Recall the MSE loss \eqref{e:MSBE} presents challenges because it is not convex for linear function approximation.    The quadratic program \eqref{e:QuadBEa} obtained from the 
 \DPLP\ motivates the variation of \eqref{e:MSBE}:
\begin{equation}
\clEBE(\theta)   =  \frac{1}{\Nsam} \sum_{k=0}^{\Nsam-1}  \bigl[    \Tdiff^\circ_{k+1}(\theta)   \bigr]^2
\label{e:MSBEconvex}
 \end{equation}
 with temporal difference defined by a modification of \eqref{e:TDonline}:
\begin{equation} 
\Tdiff^\circ_{k+1}(\theta) \eqdef   -Q^\theta (x(k),u(k) )   +   c(x(k),u(k))  +   J^\theta (x(k+1)) 
\label{e:TDonline2}
\end{equation} 
The algorithms are designed so that $J^\theta$ approximates $\uQ^\theta$, and hence $\Tdiff^\circ_{k+1}(\theta)  \approx \Tdiff_{k+1}(\theta) $
(recall from below   \eqref{e:WeAreQ}, $ \uQ(x) = \min_u Q(x,u)$ for any function $Q$).

\smallbreak
 
 The first of several versions of ``CQL'' involves a Galerkin relaxation of the constraints in the \DPLP\ \eqref{e:ConvexBE}. 
\choreS{Struggling with acronyms}  
   This requires specification of two   vector valued sequences $\{\elig_k,\elig_k^+\}$  based on the data, and denote
\begin{subequations}
\begin{align}
\zBE(\theta) & =  \frac{1}{\Nsam} \sum_{k=0}^{\Nsam-1}   \Tdiff^\circ_{k+1}(\theta)   \elig_k  
\label{e:ze}
\\
z^+(\theta) & =  \frac{1}{\Nsam} \sum_{k=0}^{\Nsam-1}  \bigl[  J^\theta (x(k))  -  Q^\theta (x(k),u(k) )    \bigr] \elig_k ^+
\label{e:z+a}
\end{align}
\end{subequations}%
with temporal difference sequence $\{  \Tdiff^\circ_{k+1}(\theta) \}$ defined in  \eqref{e:TDonline2}.
It is assumed that $\elig_k$ takes values in $\Re^d$, and that
 the entries of the vector $ \elig_k ^+$ are non-negative for each $k$. 
 \notes{need to remove equality constraints earlier}

%linprog Linear programming.
%    X = linprog(f,A,b) attempts to solve the linear programming problem:
% 
%             min f'*x    subject to:   A*x <= b
 
 \begin{programcode}{LP Convex Q-Learning}%
 \begin{subequations}
\begin{align} 
\theta^\ocp = \argmax_\theta
    \ \  &   \langle \mu,J^\theta \rangle      
\\
\st  \ \  &  \zBE(\theta) = \Zero &&
\label{e:LPQ2=Galerkin}
\\
&    z^+(\theta) \le \Zero 
\label{e:LPQ2+Galerkin}
\end{align}    
\label{e:LPQ2G}
\end{subequations}%  
\end{programcode}

This algorithm is introduced mainly because it is the most obvious translation of the general \DPLP~\eqref{e:ConvexBE}.  A preferred algorithm described next is motivated by the quadratic program \eqref{e:QuadBEa}, which directly penalizes Bellman error.  
The objective function is modified to include the empirical mean-square error \eqref{e:MSBEconvex} and a second loss function:
\begin{subequations}
\begin{align}
&
\clE^+(\theta)   =  \frac{1}{\Nsam} \sum_{k=0}^{\Nsam-1}  \bigl[   \{  J^\theta (x(k))  -  Q^\theta (x(k),u(k) )  \}_+   \bigr]^2
\label{e:L+}
\\
\hspace{-1em}
\textit{or} \quad &\clE^+(\theta)   =  \frac{1}{\Nsam} \sum_{k=0}^{\Nsam-1}  \bigl[   \{  J^\theta (x(k))  -  \uQ^\theta (x(k) )  \}_+   \bigr]^2
\label{e:L++}
\end{align}%
\label{e:L+++}%
\end{subequations}%
where $\{z\}_+ = \max(z,0)$.        The second option \eqref{e:L++}  more strongly penalizes deviation from the constraint $ \uQ^\theta \ge J^\theta $.   The choice of definition \eqref{e:L+} or  \eqref{e:L++} will depend on the relative complexity, which is application-specific.
\saveForBook{Lasso concepts would suggest no square, but we don't really care about meeting the constraints exactly}

\begin{programcode}{Convex Q-Learning}%
For positive scalars   $\kappaBE$ and $\kappa^+$, and a tolerance $\Tol \ge 0$,
\begin{subequations}
\begin{align} 
\theta^\ocp = \argmin_\theta &
   \bigl\{   -   \langle \mu,J^\theta \rangle      +   \kappaBE  \clEBE(\theta) +      \kappa^+  \clE^+(\theta)  \bigr\}     &&
\label{e:ConvexBEa_relax}
\\
 \qquad \st \ \       &   \zBE(\theta) = \Zero &&
\label{e:ConvexBEb_relax}
\\
&    z^+(\theta) \le \Tol
\label{e:ConvexBEc_relax}
\end{align} 
\label{e:ConvexBE_relax}%
\end{subequations}% 
\end{programcode}

  To understand why  the optimization problem  \eqref{e:LPQ2G}  or \eqref{e:ConvexBE_relax}  may present challenges, consider their implementation based on a kernel.   Either of these optimization problems is a convex program.  However, due to the Representer  Theorem \cite{bis06},   
the dimension of $\theta$  is equal to the number of observations $\Nsam$.   Even in simple examples, the value of $\Nsam$ for a reliable estimate may be larger than one million.   
\notes{\url{https://en.wikipedia.org/wiki/Kernel_methods_for_vector_output}}
 
%Consider a simplified version of \eqref{e:ConvexBE_relax} in which   $-   \langle \mu,J^\theta \rangle $ is removed.   
The following batch RL algorithm is designed to reduce complexity,  and there are many other potential benefits \cite{langabrie12}.  The time-horizon $\Nsam$ is broken into $B$ batches of more reasonable size,  defined by the sequence of intermediate times $T_0 =0<T_1<T_2 <\cdots <T_{B-1} <T_B=N$.   Also required are a sequence of \textit{regularizers}:    $\regBCQ_n (J,Q,\theta )$ is a convex functional of $J,Q,\theta$, that may depend on $\theta_n$.
 Examples are provided below.

\begin{programcode}{Batch Convex Q-Learning}%
With $\theta_0\in\Re^d$ given, along with     a sequence of positive scalars $\{ \kappaBE_n, \kappa_n^+\}$,  define recursively, 
\begin{subequations}
\begin{align} 
\hspace{-.75em}
\theta_{n+1} &= \argmin_\theta \Bigl \{  -   \langle \mu,J^\theta \rangle      + \kappaBE_n  \clEBE_n(\theta)  +  \kappa^+_n \clE_n^+(\theta) 
		+   \regBCQ_n (J^\theta,Q^\theta ,\theta)  \Bigr\} 
\label{e:ConvexQalmostSAgen}%
\end{align} 
%\end{subequations}%
where for $0\le n\le B-1$,
% \begin{subequations}
\begin{align}   
\clEBE_n(\theta)   
&=   \frac{1}{r_n} \sum_{k=T_n}^{T_{n+1}-1}  \bigl[    -  Q^\theta (x(k),u(k) )   +   c(x(k),u(k))  +   J^{\theta} (x(k+1))  \bigr]^2
\label{e:ConvexQalmostSAloss}%
   \\[.5em]
&
	\clE_n^+(\theta)         =  \frac{1}{r_n} \sum_{k=T_n}^{T_{n+1}-1}  \bigl[     \{  J^\theta (x(k))  -  Q^\theta (x(k),u(k) )  \}_+  \bigr]^2     
  \label{e:ConvexQalmostSAza}% 
\\
\textit{or}\quad&   
	\clE_n^+(\theta)        =  \frac{1}{r_n} \sum_{k=T_n}^{T_{n+1}-1}  \bigl[     \{  J^\theta (x(k))  -  \uQ^{\theta} (x(k) )  \}_+      \bigr]^2
\label{e:ConvexQalmostSAzb}%
\end{align} 
\label{e:ConvexQalmostSAz}%
\end{subequations}%
with $r_n = 1/(T_{n+1} -T_n )$.   
\\
Output of the algorithm: $\theta_B$, to define the final approximation $Q^{\theta_B}$.
\\ 
\end{programcode}

The constraints (\ref{e:ConvexBEb_relax}, \ref{e:ConvexBEc_relax}) are relaxed in BCQL only to streamline the discussion that follows.

\notes{
\rd{Acronyms:   LPQL, CQL,  BCQL}
Trouble,  BCQ is in ``Off-Policy Deep Reinforcement Learning without Exploration'' cite{fujmegpre19}.    
}

How to choose a regularizer?   It is expected that design of $\regBCQ_n$ will be inspired by proximal algorithms, so that it will include a term of the form 
$\| \theta -\theta_n \|^2$  (most likely a weighted norm---see discussion in \Cref{s:GainBCQL}).    With a simple scaled norm, the recursion becomes
\begin{equation}
\theta_{n+1}  = \argmin_\theta \Bigl \{  -   \langle \mu,J^\theta \rangle      +  \kappaBE_n \clEBE_n(\theta)  +   \kappa^+_n \clE_n^+(\theta) + \frac{1}{\alpha_{n+1}} \half     \| \theta -\theta_n \|^2\Bigr\}
\label{e:ConvexQalmostSA} 
\end{equation}
where $\{\alpha_n\}$ plays a role similar to the step-size in stochastic approximation.   

 Convergence in this special case is established in the the following result, based on the   steady-state expectations (recall \eqref{e:QSAergodic}):  
 \begin{subequations}
\begin{align} 
\barclEBE(\theta)   
&= \Expect_\varpi \bigl[  \{ -  Q^\theta (x(k),u(k) )   +   c(x(k),u(k))  +   J^{\theta} (x(k+1))  \bigr\}^2  \bigr]
 \label{e:ConvexQ_lossBE}%
\\
\bar \clE^+(\theta) &=
 \Expect_\varpi[ \{  J^\theta (x(k))  -  Q^\theta (x(k),u(k) )  \}_+^2     ] 
 \label{e:ConvexQ_loss+}%
\end{align}%
\label{e:ConvexQ_losses}%
\end{subequations}

\begin{proposition}
\label[proposition]{t:ConvexDQNconverges}  
Consider the BCQL algorithm \eqref{e:ConvexQalmostSA} subject to the following assumptions:
\begin{romannum}
\item   The   parameterization $(J^\theta, Q^\theta)$ is linear,    and $ \barclEBE   $ is strongly convex. % in $\theta$.

\item The non-negative step-size sequence is of the form $\alpha_n = \alpha_1/n$, with $\alpha_1>0$.    
%\[
%\sum \alpha_n = \infty\,,\qquad
%\sum \alpha_n^2 < \infty 
%\]
\item  The parameters reach steady-state limits:  % $r$,  $\kappaBE$, $\kappa^+$:
\[
r \eqdef \lim_{n\to\infty} r_n
\,,\quad  \kappaBE \eqdef \lim_{n\to\infty} \kappaBE_n 
\,,\quad  \kappa^+ \eqdef \lim_{n\to\infty} \kappa^+_n 
\]

\end{romannum}
Then, the algorithm is consistent:   $\theta_n\to\theta^\ocp$ as $n\to\infty$, where the limit is the unique optimizer:
\begin{equation}
\theta^\ocp  = \argmin_\theta \Bigl \{  -   \langle \mu,J^\theta \rangle      +\kappaBE \barclEBE(\theta)  +   \kappa^+ \barclE^+(\theta)  \Bigr\}
\label{e:ConvexQalmostSAsoln} 
\end{equation}
\end{proposition}
The proof, contained in Appendix~\ref{app:CQL},  is based on recent results from SA theory   \cite{bor20a}.
Strong convexity of $ \barclEBE   $ is obtained by design, which is not difficult since it is a quadratic function of $\theta$:
\[
 \barclEBE  (\theta) =    \theta^\transpose   M \theta + 2 \theta^\transpose b  + k
\]
with $b\in\Re^d$,  $k\ge 0$, and
\[
M =
\Expect_\varpi\bigl[   \Upupsilon_{k+1} \Upupsilon_{k+1}^\transpose   \bigr]  \,,\qquad  \textit{with}\quad  \Upupsilon_{k+1}= \psi(x(k),u(k) )  -  \psi^J(x(k+1))
\]

Justification of \eqref{e:MMQ} requires that we consider a special case of either CQL or BCQL.  
Consider the following version of BCQL in which we  apply a special parameterization:
\begin{equation}
Q^\theta (x,u) = \theta^\transpose \psi(x,u) = J^\theta(x) + A^\theta(x,u) \,, 
\label{e:AdvantagePar}
\end{equation}
 with $\theta$ constrained to a convex set $\Theta$, chosen so that $A^\theta(x,u)\ge 0$ for all $\theta\in\Theta$, $x\in\state$, $u\in\ustate$ (further discussion on this approximation architecture is contained in \Cref{s:advantage}).
We maintain the definition of $ \clEBE_n$   from \eqref{e:ConvexQalmostSAloss}.   
We also bring back \textit{inequality} constraints on the temporal difference, consistent with the \DPLP\ constraint \eqref{e:ConvexBEb}.    The batch version of  \eqref{e:ze}   is denoted
\begin{equation}
\zBE_n(\theta)   =  \frac{1}{r_n} \sum_{k=T_n}^{T_{n+1}-1}    \Tdiff^\circ_{k+1}(\theta)   \elig_k  
\label{e:zeBatch}
\end{equation}
We require   $ \elig_k  (i) \ge 0$ for all $k,i$ to  ensure that the inequality constraint  $\zBE_n(\theta)   \ge 0$ is consistent with \eqref{e:ConvexBEb}

With $\theta_0\in\Re^d$ given, along with a  positive scalar $ \kappaBE$, consider the primal dual variant of BCQL (pd-BCQL):
 \begin{subequations}
\begin{align} 
\theta_{n+1} &= \argmin_{\theta\in\Theta} \Bigl \{  -   \langle \mu,J^\theta \rangle      + \kappaBE  \clEBE_n(\theta)   
		-  \lambda_n^\transpose  \zBE_n(\theta)  
		+  \frac{1}{\alpha_{n+1}} \half     \| \theta -\theta_n \|^2 \Bigr\}
\label{e:ConvexQalmostSAcor_theta}
\\
\lambda_{n+1}  & = \bigl[  \lambda_n -   \alpha_{n+1}  \zBE_n(\theta)     \bigr]_+
\label{e:ConvexQalmostSAcor_l}
\end{align} 
\label{e:ConvexQalmostSAcor}%
\end{subequations}%
where the subscript ``$+$'' in the second recursion is a component-wise projection:  for each $i$,  the component $\lambda_{n+1} (i)$ is constrained to an interval $[0, \lambda^{\text{max} } (i)]$.

\notes{
Slater's condition:   $  \barzBE_n(\theta)  > 0$ for some $\theta$.   Discussion p. 225 Luenberger, Example 1. 
}

\begin{corollary}
\label[corollary]{t:ConvexQ_limit_cor}
Suppose that   assumptions (i) and (ii) of \Cref{t:ConvexDQNconverges}  hold.  In addition, assume that
$A^\theta(x,u) = Q^\theta(x,u) - J^\theta(x) $ is non-negative valued for $\theta\in\Theta$,   where $\Theta$ is a   polyhedral cone with non-empty interior.
Then,  
 the sequence $\{\theta_n,\, \lambda_n \}$ obtained from the pd-BCQL algorithm \eqref{e:ConvexQalmostSAcor} is convergent to a pair $(\theta^\ocp,\lambda^\ocp)$, 
 and the following hold:
\begin{romannum}
\item 
 The limit $\theta^\ocp$ is the solution of the convex program \eqref{e:MMQ},   and $\lambda^\ocp$ is the Lagrange multiplier for the linear inequality constraint in \eqref{e:MMQ}.

\item
Consider the special case:   the state space and action space are finite,  $\mu$ has full support,  and the dimension of $\elig_k$ is equal to $d_\elig = |\state|\times|\ustate|$, with
\[
\elig_k(i) = \ind \{  (x(k),u(k)) = (x^i,u^i) \}  
\]
where $\{  (x^i,u^i) \}$ is an enumeration of all state action pairs.   Suppose moreover that $(J^\oc,Q^\oc)$ is contained in the function class.      
Then,  
\[
(J^\theta,Q^\theta) \Big|_{\theta =\theta^\ocp} = (J^\oc,Q^\oc)
\]  
\qed
\end{romannum}
\end{corollary}

\Cref{t:ConvexQ_limit_cor} is a corollary to \Cref{t:ConvexDQNconverges}, in the sense that it follows the same proof for the joint sequence $\{\theta_n,\lambda_n\}$. 
Both the proposition and corollary start with a proof that $\{\theta_n\}$ is a bounded sequence,  using the  ``Borkar-Meyn'' Theorem \cite{bormey00a,bor20a,rambha17,rambha18}  (based on a scaled ODE).  
The cone assumption on $\Theta$ is imposed to simplify the proof that the algorithm is stable.    Once boundedness is established, the next step is to show that the algorithm \eqref{e:ConvexQalmostSAcor} can be approximated by a primal-dual ODE for the  saddle point problem
\begin{equation}
\max_{\lambda\ge 0} \min_{\theta\in\Theta}  \barL(\theta,\lambda)\,,  \qquad \barL(\theta,\lambda) \eqdef
 -   \langle \mu,J^\theta \rangle      +  \kappaBE \barclEBE(\theta)   
		-  \lambda^\transpose   \barzBE_n(\theta)    
\label{e:SA_saddle}
\end{equation}
The function $\barL$ is quadratic and strictly convex in $\theta$, and linear in $\lambda$.
 A suitable choice for the upper bound $\lambda^{\text{max}}$ can be found through inspection of this saddle-point problem.

This approximation suggests improvements to the  algorithm.   For example, the use of an augmented Lagrangian to ensure strict convexity in $\lambda$.
We might also make better use of
the solution to the quadratic program \eqref{e:ConvexQalmostSAcor_theta} to improve estimation of  $\lambda^\ocp$.   

\smallskip

The choice of regularizer is more subtle when we consider kernel methods, so that the ``parameterization'' is infinite dimensional.  
Discussion on this topic is postponed to \Cref{s:kernel}.

\subsection{Comparisons with Deep Q-Learning}
\label{s:DQN}

 The \textit{Deep Q Network} (DQN) algorithm was designed for   neural network   function approximation;  
the term ``deep'' refers to a large number of hidden layers.     The basic algorithm is summarized below, without imposing any particular form for $Q^\theta$.
The definition of $\{T_n \}$ is exactly as in the BCQL algorithm.  
\mindex{Deep Q Network} 
\mindex{DQN} 
 
\notes{Message to Prashant June 15: to the best of my understanding,  (5.27) is what is proposed in DQN, with neural network function approximation.   It is possible they are more clever with the regularizer.   ...  }

\begin{programcode}{DQN}%
With $\theta_0\in\Re^d$ given, along with     a sequence of positive scalars $\{\alpha_n\}$,  define recursively,
\begin{subequations}%
\begin{equation}
\theta_{n+1} = \argmin_\theta    \Bigl \{  \kappaBE \clEBE_n(\theta)  +   \frac{1}{\alpha_{n+1}}       \| \theta -\theta_n \|^2\Bigr\}
\label{e:DQN}%
\end{equation}%
where for each $n$:
\begin{equation}  
\clEBE_n(\theta)    =   \frac{1}{r_n}  \sum_{k=T_n}^{T_{n+1}-1}  \bigl[    -  Q^\theta (x(k),u(k) )   +   c(x(k),u(k))  +   \uQ^{\theta_n} (x(k+1))  \bigr]^2
\label{e:DQN_Error}%
\end{equation}%  
\label{e:DQN_All}%
\end{subequations}%
with $r_n = 1/(T_{n+1} -T_n )$.   
\\ 
Output of the algorithm: $\theta_B$,  to define the final approximation $Q^{\theta_B}$.
\\
\end{programcode}

The elegance and simplicity of DQN is clear. Most significant: if $Q^\theta$ is defined via linear function approximation,  then the minimization \eqref{e:DQN}
 is the unconstrained minimum of a quadratic.
 
 DQN appears to be nearly identical to  \eqref{e:ConvexQalmostSA}, except that  the variable $J^\theta$ does not appear in the DQN loss function.  \Cref{t:dumbDQN}  (along with
 \Cref{t:ConvexDQNconverges}  and its corollary)    show that this resemblance is superficial---\textit{the potential limits of the algorithms are entirely different}.

\begin{proposition}
\label[proposition]{t:dumbDQN}
Consider the DQN algorithm with possibly nonlinear function approximation, and with $\elig_k =   \nabla_\theta Q^{\theta} (x(k),u(k) ) \Big|_{\theta =\theta_k}$.
%\[
%\sum \alpha_n =\infty\,,\qquad \sum \alpha_n^2 <\infty
%\]
Assume that $Q^\theta$ is continuously differentiable, and its gradient $\nabla Q^\theta(x,u)$ is globally Lipschitz continuous, with  Lipschitz constant independent of $(x,u)$.   
Suppose that  $B = \infty$,  the non-negative step-size sequence satisfies  $\alpha_n = \alpha_1/n$, with $\alpha_1>0$,   and suppose that the sequence $\{\theta_n \}$ defined by the DQN algorithm is convergent to some $\theta_\infty\in\Re^d$.   

Then, this limit is a    solution to  \eqref{e:WatkinsRelax}, and moreover the algorithm admits the ODE approximation $\ddt \odestate_t =\barf(\odestate_t)$  using the vector field $\barf$ defined in  \eqref{e:WatkinsRelax}.
\qed
\end{proposition}

The assumption on the step-size is to facilitate a simple proof.  This can be replaced by the standard assumptions:
\[
\sum \alpha_n =\infty\,,\qquad \sum \alpha_n^2 <\infty
\]
 The proof of \Cref{t:dumbDQN}  can be found in Appendix~\ref{app:CQL}.  Its conclusion  should raise a warning, since  we do not know if  \eqref{e:WatkinsRelax} has a solution, or if a solution has desirable properties. The conclusions are very different for convex Q-learning. 
 
The interpretation of  \Cref{t:dumbDQN} is a bit different with the use of an $\epsy$-greedy policy for exploration.  In this case,  \eqref{e:WatkinsRelax} is close to the fixed point equation for TD($0$) learning, for which there is substantial theory subject to linear function approximation --- see \cite{sutbar18} for history.   This theory implies existence   of a solution for small $\epsy$, but isn't satisfying with regards to interpreting the solution.  Moreover, stability of parameter-dependent exploration remains a research frontier \cite{karbha16}.

\section{Implementation Guides and Extensions}
\label{s:junkyard}

\subsection{A few words on constraints}   
\label{s:advantage}

The pd-BCQL algorithm defined in \eqref{e:ConvexQalmostSAcor} is motivated by a relaxation of the inequality constraint  \eqref{e:ConvexBEb} required in the  \DPLP: 
\begin{equation}
 \zBE(\theta) \ge  -\Tol
\label{e:zBErelax}
\end{equation} 
Recall that this is a valid relaxation only if we  impose positivity on the entries of $\elig_k$.   In experiments it is found that imposing hard constraints in the convex program is more effective than the pure penalty approach used in BCQL.

With the exception of pd-BCQL,  the CQL algorithms introduce data-driven relaxations of the constraint \eqref{e:ConvexBEc} in the  \DPLP:  consider the  constraint \eqref{e:LPQ2+Galerkin} or \eqref{e:ConvexBEc_relax},  or the penalty $\clE_n^+(\theta) $ in BCQL.   A data driven approach may be convenient,  \textit{but it is unlikely that this is the best option.}   The purpose of these constraints is to enforce non-negativity of the difference,  $A^\theta(x,u) = Q^\theta(x,u) - J^\theta(x) $ (known as the advantage function in RL).

There are at least two options to enforce or approximate the inequality $A^\theta\ge0$:
\\
1.        Choose a parameterization, along with constraints on $\theta$, so that non-negativity of $A^\theta$ is automatic.   Consider for example a linear parameterization in which   $d = d^J + d^Q$,   and with basis functions $\{ \psi^J,\psi^A\}$ satisfying the following constraints:  
\begin{alignat*}{6}
%\begin{align*}
\psi_i^J(x)  &= 0  && \text{ for all $x$, and all $i> d^J$}
   \\
\psi^A_i(x,u)  &= 0   && \text{ for all $x,u$, and all $i\le d^J$} 
\end{alignat*}
subject to the further constraint that $ \psi^A_i(x,u) \ge 0$ for all $i,x,u$.    

We then define $\psi = \psi^J +\psi^A$,       $J^\theta(x)=\theta^\transpose \psi^J(x,u) $,  $Q^\theta(x,u)=\theta^\transpose \psi(x,u) $, and $A^\theta(x,u)=\theta^\transpose \psi^A(x,u) $, so that  for  any $\theta\in\Re^d$,
\[
Q^\theta (x,u) =  J^\theta(x) + A^\theta(x,u) 
\]
It follows that $Q^\theta (x,u) \ge J^\theta(x) $ for all $x,u$, provided we impose the constraint  $\theta_i \ge 0$ for $i> d^J$.  Using this approach,  the penalty term $ \kappa^+  \clE^+(\theta)      $ may be eliminated from  \eqref{e:ConvexBE_relax}, as well as the constraint $z^+(\theta) \le \Tol$.
It was found that this approach was most reliable in experiments conducted so far,  along with the relaxation \eqref{e:zBErelax},  since the algorithm reduces to a convex program (much like DQN).

\smallbreak

\noindent
2.     
   Choose a grid of points   $ G\subset \state\times \ustate$,   and  replace \eqref{e:ConvexBEc_relax} with the simple inequality constraint 
\[
\text{$ A^\theta(x^i,u^i) \ge  - \Tol$ for $(x^i,u^i)\in G$.}
\]
For example, this approach is reasonable for the LQR problem and similar control problems that involve linear matrix inequalities (such as  \Cref{e:ConvexLQRc}).  In particular,    the constraints in \eqref{e:ConvexLQR} cannot be captured by a linear parameterization of the matrices $M$,   so application of approach~1 can only approximate a subset of the constraint set.

\subsection{Gain selection and SA approximations}
\label{s:GainBCQL}

Consider the introduction of a weighted norm in \eqref{e:ConvexQalmostSA}:
\begin{equation} 
\begin{aligned}
\theta_{n+1}  &= \argmin_\theta \bigl \{  \clE_n(\theta)
+ \frac{1}{\alpha_{n+1}} \half  \| \theta -\theta_n \|^2_{W_n}  \bigr\}
\\
 \clE_n(\theta) &= 
-   \langle \mu,J^\theta \rangle      + \kappaBE_n \clEBE_n(\theta)  +  \kappa^+_n \clE_n^+(\theta) 
\end{aligned} 
\label{e:ConvexQalmostSA_preZap}%
\end{equation}
where $\| \vartheta \|^2_{W_n}  = \half \vartheta^\transpose W_n \vartheta$ for $\vartheta \in\Re^d$, 
with $W_n>0$.   

The recursion can be represented in a form similar to stochastic approximation (SA):

\begin{lemma}
\label[lemma]{t:BCQLSA}
Suppose that $\{ \clE_n (\theta) \}$ are continuously differentiable in $\theta$. Then, the parameter update in BCQL is the solution to the fixed point equation:
\begin{equation}
\theta_{n+1}  = \theta_n  -    \alpha_{n+1}  W_n^{-1} \nabla  \clE_n(\theta_{n+1}) 
\label{e:BCQLSAa}
\end{equation}
Suppose in addition   that $ \nabla  \clE_n$ is Lipschitz continuous (uniformly in $n$),
and  the sequences $\{\theta_n \}$ and $\{ \trace(   W_n^{-1}) \}$ are uniformly bounded.
Then, 
\begin{equation}
\theta_{n+1}  = \theta_n  -   \alpha_{n+1} \{  W_n^{-1} \nabla  \clE_n(\theta_n)   + \epsy_{n+1} \} 
\label{e:BCQLSA}
\end{equation}
where $ \| \epsy_{n+1} \| = O( \alpha_{n+1} )$.  
\end{lemma}

\begin{proof}
The fixed point equation \eqref{e:BCQLSAa}
follows from the first-order condition for optimality:
\[
\Zero 			= \nabla  \bigl \{  \clE_n(\theta)
+ \frac{1}{\alpha_{n+1}} \half  \| \theta -\theta_n \|^2_{W_n}  \bigr\} \big|_{\theta = \theta_{n+1}}  
		=\nabla  \clE_n(\theta_{n+1})  + \frac{1}{\alpha_{n+1}} \half W_n [\theta_{n+1} -\theta_n ]
\]

To obtain the second conclusion,     note that \eqref{e:BCQLSAa}
 implies  $\| \theta_{n+1}  - \theta_n  \| = O (\alpha_{n+1})$  under the boundedness assumptions, and then Lipschitz continuity of $\{\nabla  \clE_n\}$ in $\theta$ then gives the desired representation \eqref{e:BCQLSA}.
\end{proof}

The Zap SA algorithm of \cite{devbusmey20} is designed for recursions of the form \eqref{e:BCQLSA} so that 
\[
 W_n  \approx     \nabla^2  \barclE(\theta_n)  
\]
in which the bar designates a steady-state expectation  (recall \eqref{e:QSAergodic}). 

For the problem at hand, this is achieved in the following steps.   First, define for each $n$,
\[
A_{n+1} =   \nabla^2  \clE_n\, (\theta_n)
\]
For the linear parametrization using   \eqref{e:ConvexQalmostSAza} we obtain simple expressions.  To compress notation, denote
\[
\psi_{(k)} = \psi(x(k),u(k) ) \,, \qquad 
\psi^J_{(k)} = \psi^J(x(k)) \,. 
\]
so that \eqref{e:ConvexQalmostSA_preZap} gives
\[
\begin{aligned}
\nabla  \clE_n\, (\theta)
  =  2 \frac{1}{r_n} \sum_{k=T_n}^{T_{n+1}-1} \Big\{ &  -   \langle \mu, \psi^J \rangle  + \kappaBE_n \Tdiff^\circ_{k+1}(\theta)   \bigl[  \psi^J_{(k+1)} -  \psi_{(k)}   \bigr] 
   \\
            & + \kappa_n^+   \{   J^\theta (x(k))  -  \uQ^{\theta} (x(k) )\}_+   \bigl[ \psi^J_{(k)}  -  \psi_{(k)}    \bigr]    \Bigr\}
\end{aligned} 
\]
where $\langle \mu, \psi^J \rangle $ is the column vector whose $i$th entry is $\langle \mu, \psi^J_i \rangle $.  
Taking derivatives once more gives
\[
\begin{aligned}
A_{n+1}   = 2  \frac{1}{r_n} \sum_{k=T_n}^{T_{n+1}-1} \Big\{ &  \bigl[  \psi_{(k)}  -   \psi^J_{(k+1)}  \bigr]\bigl[  \psi_{(k)}  -   \psi^J_{(k+1)}  \bigr]^\transpose
   \\
            & + \kappa_n^+  \bigl[  \psi_{(k)}  -   \psi^J_{(k)}  \bigr]\bigl[  \psi_{(k)}  -   \psi^J_{(k)}  \bigr]^\transpose  \ind\{   J^\theta (x(k))  >  \uQ^{\theta} (x(k) )\}    \Bigr\}
\end{aligned} 
\]
We then take $W_0>0$ arbitrary, and for $n\ge 0$,
\begin{equation}
W_{n+1} = W_n + \beta_{n+1}  [   A_{n+1} - W_n]
\label{e:ZapGain}
\end{equation}
in which the step-size for this matrix recursion is relatively large:
\[
\lim_{n\to\infty}  \frac{\alpha_n}{\beta_n} = 0
\]
In \cite{devbusmey20} the choice $\beta_n = \alpha_n^\eta$ is proposed, with $\half <\eta < 1$.

The original motivation in \cite{devmey17a,devmey17b,dev19} was to minimize algorithm variance 
It is now known that the ``Zap gain''  \eqref{e:ZapGain}
 often leads to a stable algorithm, even for nonlinear function approximation (such as neural networks) \cite{chedevbusmey19}.

It may be advisable to simply use the recursive form of the batch algorithm: make the change of notation $\haA_n = W_n$ (to highlight the similarity with the Zap algorithms in \cite{devbusmey20}), and define recursively   
\[
\begin{aligned}
\theta_{n+1}  & = \theta_n  -    \alpha_{n+1}   \haA_n^{-1} \nabla  \clE_n(\theta_n)    
   \\
\haA_{n+1} & = \haA_n + \beta_{n+1} [ A_{n+1} - \haA_n]
\end{aligned} 
\]
In preliminary experiments it is found that this leads to much higher variance, but this may be offset by the reduced complexity.

\notes{commented out all the clutter}
%
%
%\begin{subequations}
%\begin{align}
%\zBE(\theta) & =    \frac{1}{r_n} \sum_{k=T_n}^{T_{n+1}-1}   \Tdiff^\circ_{k+1}(\theta)   \elig_k  
%\label{e:ze_batch}
%\\
%z^+(\theta) & =     \frac{1}{r_n} \sum_{k=T_n}^{T_{n+1}-1}  \bigl[  J^\theta (x(k))  -  Q^\theta (x(k),u(k) )    \bigr] \elig_k ^+
%\label{e:z+a_batch}
%\end{align}
%
%
%
%
% 
%
% 
%And if you really want to annoy your co-author, you can propose \bl{Zap primal-dual BCQL:}
%\begin{align}
%\theta_{n+1}  & = \theta_n  -    \alpha_{n+1}   \haA_n^{-1}\bigl \{    \nabla  \clE_n(\theta_n)    - \lambdaBE_n\nabla \zBE(\theta_n ) - \lambda_n^+\nabla z^+(\theta_n )   \bigr\}
%   \\
%\haA_{n+1} & = \haA_n + \beta_{n+1} [ A_{n+1} - \haA_n]
%\\[.5em]
% \lambdaBE_{n+1} & =  \lambdaBE_n  -  \aBE_{n+1}  \zBE(\theta_n)
% \\
% \lambda^+_{n+1} & =     \max(0, \lambda^+_n -  a^+_{n+1}  [z^+(\theta_n) -\Tol ])
%\end{align} 
%\label{e:P-D_BCQL}
%\end{subequations}%

\subsection{BCQL and kernel methods}
\label{s:kernel}

The reader is referred to other sources, such as \cite{bis06},  for the definition of a  reproducing kernel Hilbert space (RKHS) and surrounding theory.    In this subsection, the Hilbert space $\clH$ defines a two dimensional  function class that defines approximations $(J,Q)$.
\choreS{Ask Jose Principe for his fav reference  -- done, but not much help}
One formulation of this method is to choose a kernel $\Kern$ on  $  (\state\times\ustate)^2$,   in which $\Kern((x,u), (x',u'))$ is a symmetric and positive definite matrix for each $ x,x'\in\state$ and $u,u'\in\ustate$.    The function  $J$ does not depend on $u$, so for $(f,g)\in\clH$ we associate $g$ with  $Q$,  but take
\[
J(x) = f(x,u^\circ)
\]
for some distinguished $u^\circ\in\ustate$.   

A candidate regularizer in this case is
\begin{equation}
 \regBCQ_n (J,Q) =  \frac{1}{\alpha_{n+1}} \half     \| (J,Q) - (J^n,Q^n) \|_{\clH}^2
\label{e:regKernel}
\end{equation}
where $ (J^n,Q^n) \in\clH$ is the estimate at stage $n$ based on the kernel BCQL method.

The notation must be modified in this setting:  
\begin{programcode}{Kernel Batch Convex Q-Learning}%
With $(J^0,Q^0) \in\clH$ given, along with     a sequence of positive scalars $\{ \kappaBE_n, \kappa_n^+\}$,  define recursively, 
 \begin{subequations}
\begin{align} 
\hspace{-.75em}
(J^{n+1}, Q^{n+1}) &= \argmin_{J,Q} \Bigl \{  -   \langle \mu,J \rangle      +\kappaBE_n \clEBE_n(J,Q)  +  \kappa^+_n \clE_n^+(J,Q) 
		+   \regBCQ_n (J,Q )  \Bigr\} 
\label{e:ConvexQalmostSAgen_ker}%
\end{align}%
%\end{subequations}%
where for $0\le n\le B-1$,
%\begin{subequations}
\begin{align}   
\clEBE_n(J,Q)   
&=   \frac{1}{r_n} \sum_{k=T_n}^{T_{n+1}-1}  \bigl[    -  Q (x(k),u(k) )   +   c(x(k),u(k))  +   J (x(k+1))  \bigr]^2
\label{e:ConvexQalmostSAloss_ker}%
   \\[.5em]
&
	\clE_n^+(J,Q)         =  \frac{1}{r_n} \sum_{k=T_n}^{T_{n+1}-1}  \bigl[     \{  J (x(k))  -  Q (x(k),u(k) )  \}_+  \bigr]^2     
  \label{e:ConvexQalmostSAza_ker}% 
\\
\textit{or}\quad&   
	\clE_n^+(J,Q)        =  \frac{1}{r_n} \sum_{k=T_n}^{T_{n+1}-1}  \bigl[     \{  J (x(k))  -  \uQ (x(k) )  \}_+      \bigr]^2
\label{e:ConvexQalmostSAzb_ker}%
\end{align} 
\label{e:ConvexQalmostSAz_ker}%
\end{subequations}

\end{programcode}

The regularizer  \eqref{e:regKernel} is chosen so that we can apply the Representer Theorem to solve this infinite-dimensional optimization problem \eqref{e:ConvexQalmostSAz_ker}. 
The theorem states that computation of $ (J^n,Q^n) $ reduces to a finite dimensional setting, with a linearly parameterized family similar to \eqref{e:linParFamily}. 
However, the ``basis functions'' $\psi^J$ and $\psi$ depend upon $n$:   for some $\{ \theta^{n\ocp}_i \}\subset \Re^2$,
\[
\begin{aligned} 
J^n(x)  &= \sum_i  \{ { \theta_i^{n \oc}}^\transpose\Kern((x_i, u_i) ,  (x, u^\circ))   \}_1
\\
Q^n(x,u) &  = \sum_i  \{ { \theta_i^{n \oc}}^\transpose\Kern((x_i, u_i) ,  (x, u))   \}_2  
\end{aligned} 
\]
 where $\{x_i,u_i \}$ are the state-input pairs observed on the time interval $\{ T_{n-1} \le k  < T_n\}$.

\subsection{Variations}
\label{s:var}

 \notes{
\rd{copied from book draft}
I propose we keep this in arXiv, and delete for submission to a journal.
}

The total cost problem \eqref{e:Jstar} is our favorite because the control solution comes with stability guarantees under mild assumptions.   
To help the reader we discuss here alternatives, and how the methods can be adapted to other performance objectives.

\paragraph{Discounted cost}
Discounting in often preferred in operations research and computer sicence, since ``in the long run we are all dead''.  It is also convenient because it is easier to be sure that the value function is finite valued:   with $\disc\in (0,1)$ the \textit{discount factor},
\mindex{Value function!Discounted cost}
\begin{equation}
J^\oc (x) 
= \min_{\bfmu}\sum_{k=0}^\infty  \disc^k c (x(k), u(k))    \, , \quad x(0)=x\in\state\, .
\label{e:Jstar_disc}
\end{equation} 
The Q-function becomes $Q^\oc(x,u)\eqdef  c (x,u)  +  \disc  J^\oc ( \Dds(x,u) )    $, and the Bellman equation 
has the same form \eqref{e:BE_part1_gen}.

\paragraph{Shortest path problem}

Given a subset $A\subset \state$, define
\[
\tau_A = \min\{ k\ge 1 : x(k) \in A \}
\]
The discounted shortest path problem (SPP) is defined to be the minimal discounted cost incurred before reaching the set $A$:
\mindex{SPP}
\mindex{Shortest path problem}
\begin{equation}
J^\oc(x) = \min_\bfmu  \Big\{  \sum_{k=0}^{\tau_A -1}   \disc^k  c(x(k),u(k))  + J_0(x(\tau_A)) \Big\}
\label{e:SPPdet}
\end{equation}
where $J_0\colon\state\to\Re_+$ is the \textit{terminal cost}.
\mindex{Terminal cost}
For the purposes of unifying the control techniques that follow, it is useful to recast this as an instance of the total cost problem \eqref{e:Jstar}.    This requires the definition of a new state process $\bfmx^A$ with dynamics $\Dds^A$,  and a new cost function $c^A$ defined as follows:   
\begin{romannum}
\item  The modified state dynamics:   append a \textit{graveyard state} to $\grave$ to $\state$, and denote
\[
\Dds^A(x,u) =  \begin{cases}    \Dds(x,u)  &   x\in A^c 
\\	
							\grave   &  x\in A\cup\grave
			\end{cases}
\]
so that $x^A(k+i) = \grave$ for all $i\ge 1$ if $x^A(k)\in A$.
\mindex{Graveyard state}

\item   Modified cost function:
\[
c^A(x,u) =  \begin{cases}    c(x,u)  &   x\in A^c 
\\	
							J_0(x)   &  x\in A
\\
					0  &  x =\grave
\end{cases}
\]
\end{romannum}
From these definitions it follows that the value function \eqref{e:SPPdet} can be expressed
\[
J^\oc(x) = \min_\bfmu \sum_{k=0}^\infty   \disc^k  c^A(x^A(k),u(k))  
\]

Alternatively,  we can obtain a dynamic programming equation   by writing
\[
J^\oc(x) 
= \min_\bfmu \Big\{ c(x,u(0) ) + \sum_{k=1}^{\tau_A -1} \disc^k  c(x(k),u(k))   + J_0(x(\tau_A))   \Big\}
\]
with the understanding that $\sum_1^0 =0$.  The upper limit in the sum is equal to $0$ when  $\tau_A=1$;  equivalently,  $x(1)\in A$.     Hence,
\begin{equation}
\begin{aligned}
J^\oc(x) 
&
= \min_{u(0)} \Big\{ c(x,u(0) ) +   \gamma \ind\{x(1) \in A^c\} 
\Big[ \min_{u_{[1, \infty ]}}  \sum_{k=1}^{\tau_A -1}   \disc^{k-1} c(x(k),u(k)) +J_0(x(\tau_A))  \Big]
		  \Big\}   				
\\
&
= \min_{u(0)} \Big\{ c(x,u(0) ) + \gamma \ind\{x(1) \in A^c\}   J^\oc (x(1))  \Big\}  \,, \qquad x(1) = \Dds(x, u(0)) 
\\
  &=   \min_u  \Big\{ c(x,u  ) +\gamma  \ind\{ \Dds(x,u)  \in A^c\}   J^\oc ( \Dds(x,u)) \Big\}
\end{aligned} 
\label{e:SPPBE}
\end{equation}

\paragraph{Finite horizon}

Your choice of discount factor is based on how concerned you are with the distant future.   Motivation is similar for the \textit{finite horizon} formulation:  
fix a horizon $\Hor\ge 1$, and denote  
\mindex{Value function!Finite horizon}
\begin{equation}
J^\oc (x) 
= \min_{u_{[0,\Hor ]}} \sum_{k=0}^\Hor   c (x(k), u(k))    \, , \quad x(0)=x\in\state\, .
\label{e:Jstar_hor}
\end{equation} 
This can be interpreted as the total cost problem \eqref{e:Jstar}, following two modifications of the state description and the cost function, similar to the SPP:
\begin{romannum}
\item  Enlarge the state process to $x^\Hor(k) = (x(k), \hor(k))$,  where the second component is ``time'' plus an offset:
\[
\hor(k) = \hor(0) + k\,,\qquad k\ge 0
\]
\item   Extend the definition of the cost function as follows:
\[
c^\Hor((x,\hor),u) =  \begin{cases}    c(x,u)  &   \hor \le \Hor
\\	
							0   &  \hor > \Hor 
				\end{cases}
\]
\end{romannum}
That is, $c^\Hor((x,\hor),u) =c (x,u) \ind\{\hor \le \Hor\}$ for all $x,\hor, u$.  

If these definitions are clear to you, then you understand that we have succeeded in the transformation:  
\begin{equation}
J^\oc (x) 
= \min_{\bfmu}\sum_{k=0}^\infty   c^\Hor (x^\Hor(k), u(k))    \, , \quad x^\Hor(0)= (x,\hor)\,,  \  \hor = 0
\label{e:Jstar_hor_equiv}
\end{equation} 
However, to write down the Bellman equation it is necessary to consider all values of $\hor$  (at least values $\hor\le \Hor$), and not just the desired value $\hor=0$.  Letting $J^\oc (x ,\hor) $ denote the right hand side of \eqref{e:Jstar_hor_equiv} for arbitrary values of $\hor\ge 0$,  the Bellman equation  \eqref{e:BE_part1} becomes 
\begin{equation}
J^\oc (x,\hor) = \min_{u} \bigl\{ c (x,u) \ind\{\hor \le \Hor\}   +   J^\oc ( \Dds(x,u) , \hor +1) \bigr\}  
\label{e:BE_part1_hor}
\end{equation}

\saveForBook{
The similarity with \eqref{e:VIA_part1} is explored in \Cref{ex:VIA}.
}

\mindex{Value Iteration!Boundary condition}
Based on \eqref{e:Jstar_hor_equiv} and the definition of $c^\Hor$, 
we know that $J^\oc ( x , \hor) \equiv 0$ for $\hor >\Hor$.  This is considered a \textit{boundary condition} for the recursion  \eqref{e:BE_part1_hor}, which is put to work as follows:   
first, since   $J^\oc ( x , \Hor+1) \equiv 0$, 
\[
J^\oc (x,\Hor) =\uc(x) \eqdef  \min_{u}  c (x,u)  
\]
Applying \eqref{e:BE_part1_hor} once more gives,  
\[
J^\oc (x,\Hor -1) = \min_{u} \bigl\{ c (x,u)    +   \uc( \Dds(x,u)  ) \bigr\}  
\]
If the state space $\state$ is finite then these steps can be repeated until we obtain the value function $J^\oc (\varble,0)$.

What about the policy?    It is again obtained via \eqref{e:BE_part1_hor},  but the optimal input depends on the extended state, based on the policy
\[
\fee^\oc(x,\hor) =  \argmin_{u} \bigl\{ c (x,u)    +   J^\oc ( \Dds(x,u) , \hor +1) \bigr\}  \,, \qquad \hor \le\Hor\,
\]
This means that the feedback is no longer time-homogeneous:\footnote{Substituting    $\hor^\oc(k) =k$ is justified because we cannot control time!}  
\begin{equation}
u^\oc(k) = \fee^\oc(x^\oc(k), \hor^\oc(k)) = \fee^\oc(x^\oc(k), k)
\label{e:OptPolicyFiniteHor}
\end{equation}

\subsection{Extensions to MDPs}
\label{s:H-MDP}

The full extension of convex Q-learning will be the topic of a sequel.  We present here a few ideas on how to extend these algorithms to MDP models. 

The state space model \eqref{e:dis_NLSSx_controlled_opt} is replaced by the controlled Markov model,
\begin{equation}
X(k+1)   = \Dds(X(k),U(k),N(k+1) ) \,,\qquad k\ge 0
\label{e:NLSSdiscreteControl}
\end{equation}
where we have opted for upper-case to denote random variables, and $\bfmN$ is an i.i.d.\ sequence.    For simplicity we assume here that the state space and action space are finite, and that the disturbance $\bfmN$ also evolves on a finite set (without loss of generality, given the assumptions on $\state$ and $\ustate$).  The controlled transition matrix is denoted
\[
\begin{aligned}
P_u(x,x')  &=  \Prob\{X(k+1) =x' \mid X(k)=x\,, \ U(k) =u \} 
\\
& =  \Prob\{ \Dds( x, u,N(1) )  = x' \}   \,,\qquad\qquad x,x'\in\state\,,\  u\in\ustate
\end{aligned} 
\]
The following operator-theoretic notation is useful:  for any function $h\colon\state\to\Re$,  
\[
P_uh\, (x) =  \sum_{x'}P_u(x,x') h(x')  = \Expect[ h(X(k+1)) \mid X(k) =x\,,\ U(k)=u]
\]

An \textit{admissible policy} is a sequence of mappings to define the input:
\[
U(k) = \fee_k(X(0),\dots, X(k)) 
\]
This includes randomized policies or ``quasi-randomized polices'' of the form \eqref{e:randomizedStateFB}, since we can always extend the state process $\bfmX$ to include an exploration sequence. 

  As usual, a stationary policy is state feedback:  $U(k) = \fee(X(k)) $.
We impose the following controllability condition:  for each $x^0, x^1\in\state$, there is an admissible policy such that
\[
\Expect [\tau_{x^1}  \mid X(0) = x^0 ]  <\infty
\]
where $\tau_{x^1} $ is the first time to reach $x^1$.

  Denote for any admissible input sequence the average cost:
\begin{equation}
\eta_U(x) = \limsup_{n\to\infty} \frac{1}{n} \sum_{k=0}^{n-1} \Expect_x[c(X(k),U(k) )]  \,, \qquad x\in\state\,,
\label{e:avgCost}
\end{equation}
and let $\eta^\oc$ denote the minimum over all admissible inputs (independent of $x$ under the conditions imposed here \cite{lerlas96a,put14,ber12a,CTCN}).   
 The average cost problem is a natural analog of the total cost optimal control problem, since the dynamic programming equation for the \textit{relative value function} $h^\oc$ nearly coincides with  \eqref{e:BE_part1}:
 \[
 h^\oc(x) =\min_{u} \bigl\{ c (x,u)  +     P_u h^\oc\, (x)   \bigr\} -\eta^\oc \,, \qquad x\in\state 
 \]

The   Q-learning formulation of \cite{aboberbor01} is adopted here:
\begin{equation}
Q^\oc(x,u) = c(x,u) +   P_u \uQ^\oc(x)   - \delta \langle \nu \,, Q^\oc \rangle
\label{e:DCOE-Q}
\end{equation}
where $\nu$ is any pmf on $\state\times\ustate$ (a dirac delta function is most convenient), and $\delta>0$.    The normalization is imposed so that there is a unique solution to \eqref{e:DCOE-Q},  and this solution results in
\[
\delta \langle \nu \,, Q^\oc \rangle = \eta^\oc
\]
The relative value function is also not unique, but one solution is given by $h^\oc(x) = \min_u Q^\oc(x,u)$,  and the optimal policy is any minimizer   (recall \eqref{e:fee}).

There is then an obvious modification of the BCQL algorithm.  First, in terms of notation we replace $J^\theta$ by $h^\theta$,  and then modify the loss function as follows:
\begin{equation}
\clEBE_n(\theta)    =   \frac{1}{r_n} \sum_{k=T_n}^{T_{n+1}-1}  \bigl\{    -  Q^\theta (x(k),u(k) )   - \delta \langle \nu \,, Q^\theta \rangle
 +   c(x(k),u(k))  +   h^{\theta} _{k+1 \mid k}    \bigr\}^2
\label{e:ConvexQalmostSAzMarkov}
\end{equation}
where $ h^{\theta} _{k+1 \mid k} = \Expect[ h^\theta(X(k+1)) \mid X(k), U(k)]  $.   
A challenge in the Markovian setting is the approximation of the conditional expectation.  We propose three options:

\paragraph{1.  Direct computation}
   If we have access to a model, then this is obtained as a simple sum: 
\[
h^{\theta} _{k+1 \mid k}  = \sum_{k'}  P_u(x,x')  h^\theta(x') \,,\qquad \text{ on observing $(x,u) = (X(k), U(k) )$,}
\]

\paragraph{2.  Monte-Carlo}
If we do not have a model, or if the state space is large, then we might resort to computing the \textit{empirical pmf} to approximate the conditional expectation (this is one role of the  \textit{experience replay  buffer} encountered in the RL literature \cite{lin92a,langabrie12,osbroydanzhe17}).    
\saveForBook{\cite{lin92a} coined replay buffer: see p.~6 of review \cite{langabrie12}.
}

It may be preferable to opt for a   Galerkin relaxation, similar to what was used in several of the algorithms introduced for the deterministic model:
   choose $d_G$ functions 
 $\{ h_k :  1\le k\le d_G\}$,  and consider the finite-dimensional function class:  $\widehat\clH = \{ \sum_i \alpha_i h_i : \alpha\in\Re^{d_G} \}$ to define an estimate of the random variable $Z$:
 \begin{equation}
\widehat\Expect[ Z \mid \sigma(Y) ] =  \argmin_{h\in\widehat\clH} \Expect[ ( Z - h(Y))^2]
\label{e:haCE_loss}
\end{equation} 
As is well known, the solution is characterized by orthogonality:
\begin{lemma}
\label[lemma]{t:haCEMMSE} 
A function  $h^\circ = \sum_i \alpha_i^\circ h_i $ solves the minimum in \eqref{e:haCE_loss} if and only   if the following holds for each $1\le i\le d_G$:  
\begin{equation}
0 =  \Expect[ ( Z - h^\circ(Y)) h_i(Y)] 
\label{e:haCEMMSE}
\end{equation}
Any minimizer satisfies 
\begin{equation}
A \alpha^\circ = b
\label{e:haCEMMSEsoln}
\end{equation}
where $A$ is a $d_G\times d_G$ matrix and $b$ is a $d_G$-dimensional vector, 
with entries 
\begin{equation}
A_{i,j} = \Expect[  h_i(Y) h_j(Y) ]  \,,\qquad  b  = \Expect[   Z  h_i(Y)] 
\label{e:haCEMMSEsolnPars}
\end{equation}
Consequently, if $A$ is full rank, then $\alpha^\oc = A^{-1} b$ is the unique minimizer.
\qed
\end{lemma}

\paragraph{3.  Pretend the world is deterministic.}   This means we abandon the conditional expectation, and instead define a loss function similar to the deterministic setting:  
  \begin{equation}
\clEBEvar_n(\theta)    =   \frac{1}{r_n} \sum_{k=T_n}^{T_{n+1}-1}  \bigl[    -  Q^\theta (x(k),u(k) )   - \delta \langle \nu \,, Q^\theta \rangle
 +   c(x(k),u(k))  +   h^{\theta}(X(k+1) )   \bigr]^2
\label{e:ConvexQalmostSAzMarkovVar}
\end{equation}
where the inclusion of ``var'' in this new notation is explained in \Cref{t:NoisyLoss} that follows.  
 Denote the steady state loss functions
\begin{equation}
\begin{aligned} 
\barclEBE(\theta)   & =   \Expect_\varpi  \bigl[    -  Q^\theta (x(k),u(k) )   - \delta \langle \nu \,, Q^\theta \rangle
 +   c(x(k),u(k))  +   h^{\theta} _{k+1 \mid k}   \bigr\}^2  \bigr]
 \\
 \barclEBEvar(\theta)   & =   \Expect_\varpi  \bigl[    -  Q^\theta (x(k),u(k) )   - \delta \langle \nu \,, Q^\theta \rangle
 +   c(x(k),u(k))  +  h^{\theta}(X(k+1) )  \bigr\}^2  \bigr]
 \end{aligned} 
\label{e:barConvexDQNzMarkov}
\end{equation}

\begin{proposition}
\label[proposition]{t:NoisyLoss}  Given any stationary policy,  and any steady-state distribution $\varpi$ for $(X(k),U(k))$,
  the respective means of the loss function are related as follows:
\[
\begin{aligned} 
 \barclEBEvar(\theta) & = 
\barclEBE(\theta)  +   \sigma^2_{k+1\mid k} (\theta)
\\
\textit{where}\quad
 \sigma^2_{k+1\mid k} (\theta) &=  \Expect_{\varpi} \bigl[  \bigl(  h^\theta(X(k+1))   - h^{\theta} _{k+1 \mid k}  \bigr)^2 \bigr]
\end{aligned} 
\]
\qed
\end{proposition}

Consequently, if the variance of $X(k+1) - X(k)$ is not   significant, then the two objective functions  $\clEBEvar_n $ and $ 
\clEBE_n  $ are not very different.   In applications to robotics it is surely best to use the simpler loss function $\clEBEvar_n(\theta) $, while this may not be advisable in financial applications.

  %%%%%%%%%%%%%%%%%%%%%%%%%%%%%%%%%%%%%%%%%%%%%%%%%%%%%%%%%%%%%%%%%%%%%%%%%%%%%%%%
\section{Example} 
\label{s:ex}

%\subsection{Nonlinear Filtering and Mean Field Games and ... ?}

%\subsection{Mountain Car}

A simple example is illustrated in \Cref{f:MountainCar}, in which the two dimensional state space is   position and velocity:
\[
x(t)\in \state = [z^{\text{min}},  z^{\text{goal}}   ]\times [-\barv, \barv]  
\]
Where $z^{\text{min}}$ is a lower limit for $x(t)$,   and the target state is $z^{\text{goal}}$.  The velocity is bounded in magnitude by $\barv>0$.
The input $u$ is the throttle position (which is negative when the car is in reverse).   This example was  introduced in the dissertation \cite{moo90},  and has since become a favorite basic example  in the RL literature \cite{sutbar18}. 

\saveForBook{
 Within the RL literature, this example was   introduced in the dissertation \cite{moo90},  and has since become a favorite basic example \cite{sinsut96,sutbar18}.
}

  \begin{wrapfigure}[11]{R}{0.25\textwidth}
  \smallskip
\centering 
    \includegraphics[width=1\hsize]{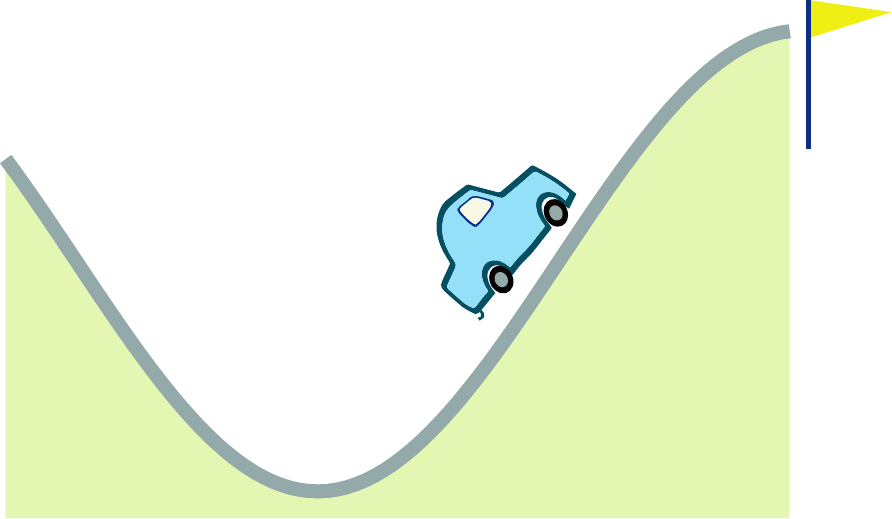}
  \caption{Mountain Car}
\label{f:MountainCar}
  \vspace{-1em}
\end{wrapfigure}

Due to state and input constraints,  a feasible policy will sometimes put the car in reverse, and travel at maximal speed away from the goal to reach a higher elevation to the left.  Several cycles back and forth may be required to reach the goal.   It is a good example to test the theory because the value function is not very smooth.

\saveForBook{
What makes this problem in interesting is that the engine is so weak, that it is impossible to reach the hill from most initial conditions:   it is sensible to hit the throttle, and head towards the goal at maximum speed.  In most cases, the car will stall before reaching the goal.   A feasible policy will put the car in reverse, and travel at maximal speed away from the goal to reach a higher elevation to the left.  Several cycles back and forth may be required to reach the goal. 
}

  \begin{wrapfigure}[11]{L}{0.25\textwidth}
  \smallskip
\centering 
    \includegraphics[width=1\hsize]{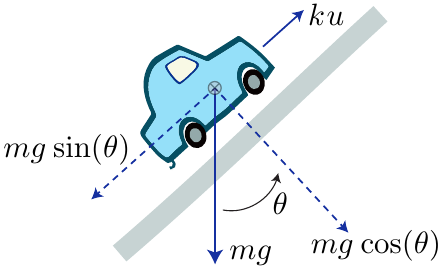}
  \caption{Two forces on the Mountain Car}
\label{f:MountainCarPhysics}
  \vspace{-1em}
\end{wrapfigure}

A continuous-time model can be constructed based on the two forces on the on the car shown in \Cref{f:MountainCarPhysics}.    With $ z^{\text{goal}}  - z(t)$ the distance along the road to the goal we obtain
\[
m a = - mg \sin(\theta) + k u
\]
where $a = \frac{d^2}{dt^2} x$,   and $\theta > 0$ for the special case shown in \Cref{f:MountainCarPhysics}.
This can be written in state space form
\[
\begin{aligned}
\ddt x_1   &=  x_2 
   \\
\ddt x_2            & =  \frac{k}{m} u  -  g \sin(\theta(x_1)) 
\end{aligned} 
\]
where $\theta(x_1)$ is the road grade when at position $x_1$.

\saveForBook{This state space model leaves much out of the navigation problem---namely, wind, traffic, and pot-holes.  
Also missing are hard constraints on position and velocity assumed at the start.   
}

A discrete time model is adopted in   \cite[Ch.~10]{sutbar18} of the form: 
\notes{May be clearer in most recent version of the book}
%\eqref{e:dis_NLSSx_controlled}:
\begin{subequations}
\begin{equation}  
\begin{aligned}
x_1(k+1)  &=   \llbracket  x_1(t) + x_2(t)     \rrbracket_1
   \\
x_2(k+1)          & =  \llbracket  x_2(k)   + 10^{-3}  u(k)  -  2.5\times10^{-3}  \cos(3 x_1(k))  \rrbracket_2
\end{aligned} 
\end{equation}%
which corresponds to 
$\theta(x_1) = \pi + 3 x_1$.  The brackets are projecting the values of $x_1(k+1) $ to the interval $ [z^{\text{min}},  z^{\text{goal}}   ] = [-1.2, 0.5]$,
and   $x_1(k+1) $ to the interval $ [-\barv, \barv]  $.   We adopt the values used in   \cite{sutbar18}:
\begin{equation}
\text{$z^{\text{min}} = -1.2$,  \ $ z^{\text{goal}} = 0.5$, \  and \  $\barv = 7\times 10^{-2}$.
}
\label{e:MountainCarPars}
\end{equation}
\end{subequations}

The control objective is to reach the goal in minimal time, but this can also be cast  as a total cost optimal control problem.   Let $x^e = ( z^{\text{goal}} , 0)^\transpose  $,   and reduce the state space so that $x^e$ is the only state $x=  ( z,v)^\transpose \in\state$  satisfying $z = z^{\text{goal}}$.  This is justified because the car parks on reaching the goal.
Let $c(x,u) = 1$ for all $x,u$ with $x\neq  x^e $,   and $c( x^e ,u) \equiv 0$.      

 The optimal total cost \eqref{e:Jstar} is finite for each initial condition, and the Bellman equation \eqref{e:BE_part1} becomes
\[
J^\oc (x) =  1 + \min_{u} \bigl\{     J^\oc ( \Dds(x,u) ) \bigr\}  \,,\quad x_1<   z^{\text{goal}} 
\]
with the usual boundary constraint $ J^\oc ( x^e) =0$.

\begin{figure}[h] 
\centering 
    \includegraphics[width=.8\hsize]{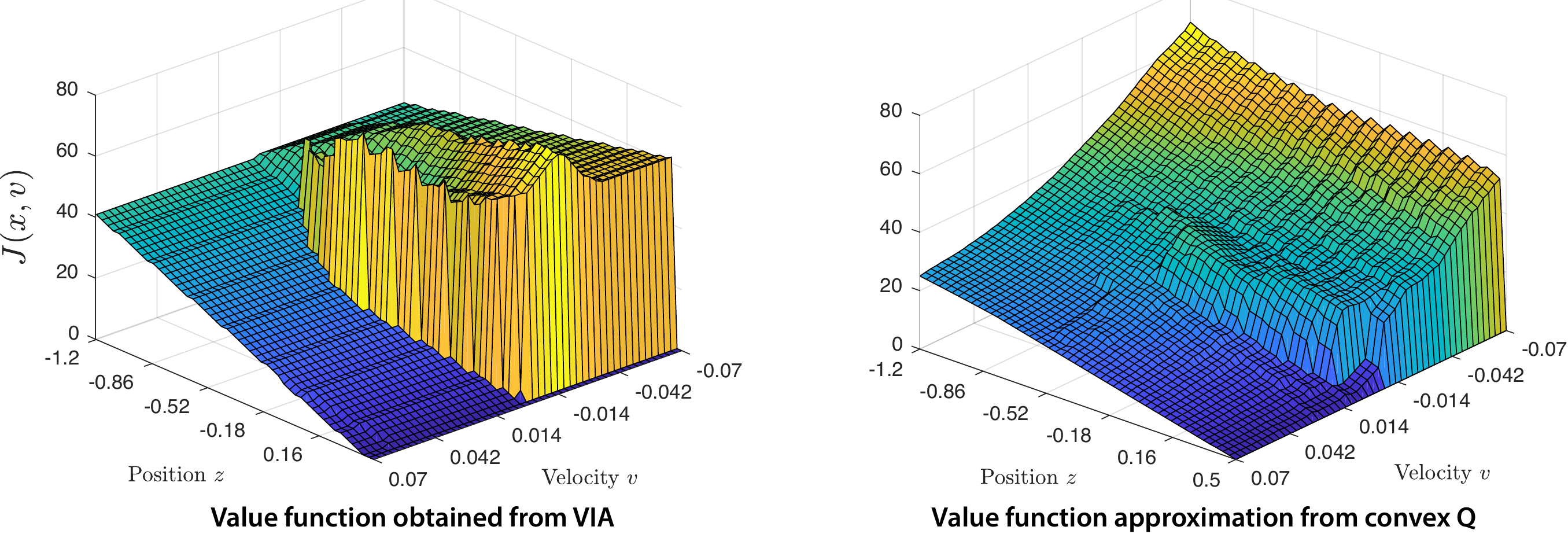}
  \caption{Value function and its approximation using a version of CQL}
\label{f:MountainCarFirstGo} 
\end{figure}

Experiments with this example are a work in progress.
\Cref{f:MountainCarFirstGo} shows results from one experiment using a parameterization of the form \eqref{e:AdvantagePar},
in which each $ \psi_i^J(x)  $ was obtained via binning:
\[
\psi_i^J(x)  = \ind\{ x \in B_i \}
\]
where the finite collection of sets is disjoint,  with $\cup B_i = \{ x\in\Re^2 :  x_1<0.5 \}$.    In this experiment they were chosen to be rectangular,   with 800 in total. 
The union was chosen to equal a half space because of the following choice for the advantage basis.   A vector $x^\Delta\in\Re^2_+$ was chosen, equal to the center of the unique bin in the positive quadrant containing the origin as one vertex.   
We then defined
\begin{alignat*}{6} 
 \psi_i^A(x,u)  &=   \ind\{u=1\, , x<0.5 \} \psi_i^J(x+ x^\Delta)  && \text{ for all $x,u$, and all $ d^J+1 \le i \le 2d^J $}
\\
\psi_i^A(x,u)  &= \ind\{u=-1\, , x<0.5 \}   \psi^j_i(x- x^\Delta)   \quad  &&\text{ for all $x,u$, and all $ 2d^J+1 \le i \le 3d^J $}
\end{alignat*}

Every algorithm failed with the basis with $x^\Delta = \Zero$ (no shift).   The explanation comes from the fact that the state moves very slowly,   which means that $x(k) $ and $x(k+1)$ often lie in the same bin.   For such $k$ we have $ J^{\theta}(x(k+1))  =   J^{\theta}(x(k))$, which is bad news:   it is easy to find a vector $\theta^0$ satisfying $Q^{\theta^0}(x,u) = 1 + J^{\theta^0}(x)$ when $z<0.5$  (for example, take $ J^{\theta}\equiv \text{const.}$ and  $Q^{\theta^0}\equiv 1+ \text{const.} $).    Hence for these bad values of $k$,
\[
Q^{\theta^0}(x(k),u(k)) = 1 + J^{\theta^0}(x(k)) =  c(x(k),u(k)) +  J^{\theta^0}(x(k+1))
\]
That is, \textit{the observed  Bellman error is precisely zero!}    Both convex Q and DQN will likely return $\theta^*\approx \theta^0$ without an enormous number of bins.   The introduction of the shift resolved this problem.   
 
Better results were obtained with an extension of this basis to include    quadratics, defined so that the span of $\psi^J_i$ includes
all quadratics satisfying $q(z,v) =0$ whenever $z=0.5$.   To achieve this, we  simply merged four bins adjacent to $z\equiv 0.5$,  and nearest to $v=0.07$,   and replaced three basis vectors with quadratics  (chosen to be non-negative on the state space).

 The plot on the left hand side of \Cref{f:MountainCarFirstGo} was obtained using value iteration for an approximate model,  in which the state space was discretized with $z$-values equally spaced at values $k\times s_z$;  
 $v$-values equally spaced at values $k\times sv$,  with   $s_z=0.041$ and $s_v=0.001$.\footnote{Many thanks to Fan Lu at UF for conducting this experiment}
 The approximation shown on the right was obtained using Convex Q-Learning \eqref{e:ConvexBE_relax} with slight modifications:  First,  the positivity penalties and constraints were relaxed since $Q^\theta\ge J^\theta$ was imposed via constraints on $\theta$.  Second,  \eqref{e:ConvexBEb_relax} was relaxed to the  inequality constraints \eqref{e:zBErelax} with $\Tol=0$,  resulting in 
 \begin{align*} 
\theta^\ocp = \argmin_\theta &
   \bigl\{   -   \langle \mu,J^\theta \rangle      +   \kappaBE  \clEBE(\theta)   \bigr\}     &&
\\
 \qquad \st \ \       &   \zBE(\theta)   \ge \Zero && 
 \\
 & \theta_i \ge 0\,,\qquad i\ge d_J
\end{align*} 
Finally, the Galerkin relaxation was essentially abandoned via
 \[
    \elig_k (i)    =     \ind\{ x(k) = x^i  ,  u(k) = u^i\}   \,,  \qquad 1\le i\le \Nsam\,,
 \]
where $\Nsam =10^4$ is the run length, and $\{  x^i , u^i: 1\le i\le N\} $ represents all values observed!   This means we are enforcing $ \Tdiff^\circ_{k+1}(\theta^*)  \ge 0$ for all $k$ over the run.  This is a very complex approach:  with $10^4$ observations we have the same number of constraints.    Regardless, the {\tt quadprog} command in Matlab returned the answer in one hour 
on a 2018 MacBook Pro with 16GB of ram.
\notes{used $10^6$ samples for quadratic objective, but no need to get into this here.  Time to increase basis dimension, and improve constraint satisfaction, perhaps with binning}

%  \begin{wrapfigure}{l}{0.35\textwidth}
%  \vspace{-1em}
%  \begin{center}
%    \includegraphics[width=1\hsize]{MountainCarValueFunction.pdf}
%  \end{center} 
%\caption[Value function for Mountain Car]{Value function for Mountain Car:   anticipated shape for states $x=(x_1,0)$  (the initial velocity is zero).
%\rd{ [to be revised]} 
%}
%\label{f:MountainCarValueFunction}
%
%  \vspace{-1em}
%\end{wrapfigure}
%
%\Cref{f:MountainCarValueFunction} shows the anticipated shape of the value function for initial conditions corresponding to the car starting at rest (corresponding to $x_2=0$).    
%The total cost is low with initial condition $x_0 = (-1.2,0)$ because the car can reach the goal without stalling.

  %%%%%%%%%%%%%%%%%%%%%%%%%%%%%%%%%%%%%%%%%%%%%%%%%%%%%%%%%%%%%%%%%%%%%%%%%%%%%%%%
\section{Conclusions} 
\label{s:conc}

The LP and QP characterization of dynamic programming equations gives rise to RL algorithms that provably convergent, and for which we know what problem we are actually solving.   Much more work is required to develop these algorithms for particular applications, and to improve efficiency through a combination of algorithm design and techniques from optimization theory.     

An intriguing open question regards algorithms for MDP models.    In this setting, a minor variant of the linear program \eqref{e:ConvexBEJ} is the \textit{dual} of Manne's LP \cite{araborferghomar93}.   The primal is expressed:
\begin{equation}
\begin{aligned}
\min \quad  &   \Expect_\varpi[ c( X,U) ] 
  \\
\st \quad            & \varpi \in\clP
\end{aligned} 
\label{e:MannePrimal}
\end{equation}
where $\clP$ is a convex subset of probability measures on $\state\times\ustate$.    Characterization of $\clP$ is most easily described when  $\state\times\ustate$ is discrete.    
In this case, for any pmf on the product space we can write via Baye's rule
\[
\varpi(x,u) = \pi(x) \fee(u \mid x)
\]
where $\pi$ is the first marginal, and $ \fee$ is   interpreted as a randomized policy.  We say that $\varpi \in\clP$  if $\pi$ is a steady-state distribution for $\bfmX$ when controlled using $\fee$.   

Actor-critic algorithms are designed to optimize average cost over a family of randomized policies $\{ \fee^\theta \}$.   The mapping from a randomized policy $ \fee^\theta$ to a bivariate pmf $\varpi^\theta \in\clP$ is highly nonlinear, which means   the valuable convexity of the primal \eqref{e:MannePrimal} is abandoned.    We hope to devise techniques to construct a convex family of pmfs $\{ \varpi^\theta \} \subset \clP$ that generate policies which   approximate the randomized polices of interest.   We then arrive at a convex program that approximates \eqref{e:MannePrimal}, and from its optimizer obtain a policy that achieves the optimal average cost:
\[
\fee^{\theta^\ocp} (u\mid x) =   \frac{1}{\pi^{\theta^\ocp} (x) }  \varpi^{\theta^\ocp} (x,u)\,,\qquad  \textit{where}\quad   \pi^{\theta^\ocp} (x) =\sum_{u'}   \varpi^{\theta^\ocp} (x,u')
\]
Existing actor-critic theory might be extended to obtain RL algorithms to estimate $\theta^\ocp$.

\bibliographystyle{abbrv}  
%\bibliography{strings,markov,q,DPLP} 
  %QSA
  
\def\cprime{$'$}\def\cprime{$'$}

%\vspace{2cm}
  
  \clearpage
   
    \addcontentsline{toc}{section}{Appendices}

  \centerline{\LARGE \bf Appendices}

  \appendix
  
  \section{Convex programs for value functions}

\begin{lemma}
\label[lemma]{t:preLPforJQ}
Suppose that  the assumptions of  \Cref{t:LPforJQ} hold:  the value function $J^\oc$ defined in \eqref{e:Jstar} is finite-valued, inf-compact, and vanishes only at $x^e$.  Then, for each $x\in\state$ and stationary policy $\fee$ for which $J^\fee(x) <\infty$,
\[
 \lim_{k\to\infty}  J^\fee( x(k)) =0\qquad   \lim_{k\to\infty}    x(k)  =x^e  
\]
\end{lemma}

\begin{proof}
With $u(k) = \fee(x(k))$ for each $k$, we have the simple dynamic programming equation:  
\[
J^\fee(x)  = \sum_{k=0}^{N-1}   c (x(k), u(k) )   +  J^\fee( x(N))
\]
On taking the limit as $N\to\infty$ we obtain
\[
J^\fee(x)  = \sum_{k=0}^\infty   c (x(k), \fee(x(k))  )   +  \lim_{N\to\infty}  J^\fee( x(N))
\]
It follows that $ \lim_{N\to\infty}  J^\fee( x(N)) =0$.  

Now, using $J^\fee \ge J^\oc$,  it follows that  $ \lim_{N\to\infty}  J^\oc( x(N)) =0$ as well.  It is here that we apply the inf-compact assumption,
along with continuity of $J^\oc$, with together imply the desired limit $  \lim_{N\to\infty}    x(N)  =x^e$.
\end{proof}  

\begin{proof}[Proof of \Cref{t:LPforJQ}]
Since $\mu$ is non-negative but otherwise arbitrary,  to prove the proposition it is both necessary and sufficient to establish the bound  $J  \le J^\oc$ for each  feasible $(J,Q)$.

The constraints \Cref{e:ConvexBEb,e:ConvexBEc} then give, for any input-state sequence,  and any feasible  $(J,Q)$,
\[
J(x(k)) \le Q(x(k), u(k))  \le c(x(k), u(k))  +   J ( \Dds(x(k), u(k)) ) = c(x(k), u(k))  +   J ( x(k+1)) 
\]
This bound can be iterated to obtain, for any $N\ge 1$,
\[
J(x)  \le  \sum_{k=0}^{N-1}   c (x(k), u(k) )   +  J( x(N)) \,,\quad x= x(0) \in \state 
\]
We now apply \Cref{t:preLPforJQ}:  Fix $x\in\state$, and a stationary policy $\fee$ for which $J^\fee(x) <\infty$.   
With $u(k) = \fee(x(k))$ for each $k$ in the preceding bound we obtain 
\[
J(x)  \le  J^\fee(x)   +  \lim_{N \to\infty}  J( x(N)) \,,\quad x= x(0) \in \state 
\]
Continuity of $J$ combined with \Cref{t:preLPforJQ} then implies that the limit on the right hand side is $J(x^e)=0$,   
giving $J(x) \le   J^\fee(x) $.   It follows that  $J(x) \le J^\oc(x)$ for all $x$ as claimed.   
\end{proof}

\notes{Prashant:  I realized I made the proof a bit more complex than necessary, since I could have used $\fee^\oc$ throughout.  This is habit from other optimal control papers where you don't initially know if the optimal policy exists!}

\begin{proof}[Proof of \Cref{t:LPforLQR}]
The reader is referred to standard texts for the derivation of the ARE  \cite{andmoo90,bro90}.  The following is a worthwhile exercise:  postulate that $J^\oc$ is a quadratic function of $x$,  and you will find that the Bellman equation implies the ARE.

Now, on to the derivation of \eqref{e:ConvexLQR}.
The variables in the linear program introduced in \Cref{t:LPforJQ} consist of  functions $J$ and $Q$.  For the LQR problem we  restrict to quadratic functions:
\[
J(x) = x^\transpose M x\,,\qquad
Q(x,u)   =z^\transpose {M^Q} z
\]
and treat the symmetric matrices $(M, M^Q)$ as variables.   \saveForBook{need an essay on quadratic functions and why they are symmetric WLOG}

To establish \eqref{e:ConvexLQR} we are left to show  1) the objective functions  \eqref{e:ConvexBEa} and   \eqref{e:ConvexLQRa} coincide for some $\mu$, and 2)
   the functional constraints  (\ref{e:ConvexBEb}, \ref{e:ConvexBEc})   are equivalent to the matrix   inequality \eqref{e:ConvexLQRc}. The first task is the simplest:  
\[
   \trace(M)   = \sum_{i=1}^n  J(e^i)  =  \langle \mu, J \rangle  
\]
with $\{e^i\}$ the standard basis elements in $\Re^n$,   and $\mu(e^i) = 1$ for each $i$. 

The equivalence of \eqref{e:ConvexLQRc} and (\ref{e:ConvexBEb}, \ref{e:ConvexBEc}) is established next, and through this we also obtain  \eqref{e:JQzQ}.    
In view of the discussion preceding \eqref{e:ConvexBEJ}, 
the inequality constraint \eqref{e:ConvexBEb} can be strengthened to equality:
 \begin{equation}
Q^\oc(x,u)  =   c (x,u)  +   J^\oc ( F x+ G u)    
\label{e:LQRQ}
\end{equation} 
It remains to establish the equivalence of  \eqref{e:ConvexLQRc} and \eqref{e:ConvexBEJ}.

Applying  \eqref{e:LQRQ}, we obtain a mapping from $M$ to $M^Q$.
Denote 
\[
M^J = 
\begin{bmatrix}
	M & 0 \\
	0 & 0 \\
\end{bmatrix} \,, \quad
\Xi=
\begin{bmatrix}
	 F& G  \\
	F & G \\
\end{bmatrix}
\]
giving for all $x$ and $z^\transpose  = (x^\transpose,  u^\transpose)$,
\[
J(x) =x^\transpose M x = z^\transpose M^J z
\,,\qquad 
J( F x+ G u)     = z^\transpose  \Xi^\transpose M^J \Xi z
 \]
This and  \eqref{e:LQRQ} gives, for any $z$,
\begin{align*}
z^\transpose M^Qz = 
	Q(x, u) 
	&= c(x, u) + J( F x+ G u)
	\\ 
	& = z^\transpose M^c z + z^\transpose \Xi^\transpose M^J \Xi z
\end{align*} 
The desired mapping from $M$ to $M^Q$ then follows, under  the standing assumption that $M^Q$ is a symmetric matrix:
\begin{align*}
	M^Q   = M^c +\Xi^\transpose M^J \Xi   =    \begin{bmatrix}  S & \Zero \\ \Zero &R \end{bmatrix} 
	+
	\begin{bmatrix}
		F^TMF & F^TMG \\
		G^TMF & G^TMG \\
	 \end{bmatrix} 
\end{align*}

The constraint \eqref{e:ConvexBEJ}  is thus equivalent to 
\begin{align*}
	z^\transpose M^J z  = J(x)  \leq Q(x, u) = z^\transpose M^Q z\,,\quad \textit{for all $z$}
\end{align*}
This is equivalent to the constraint $	M^J  \leq  M^Q $, which is  \eqref{e:ConvexLQRc}.
\end{proof}

\section{Limit theory for Convex Q learning and DQN}
\label{app:CQL}

The proof of \Cref{t:ConvexDQNconverges} and surrounding results is based on an ODE approximation for the parameter estimates, which is made possible by  (A$\bfqsaprobe$).  

The step-size assumption in \Cref{t:ConvexDQNconverges} is introduced to simplify the following ODE approximation.

\begin{lemma}
\label[lemma]{t:LLNuni}
The following conclusions hold under the assumptions of  \Cref{t:ConvexDQNconverges}:
\begin{romannum}
\item 
For any $n,n_0\ge 1$,  and any function $g\colon\zstate\to\Re$,
\[
\sum_{k=n}^{n+n_0 -1}   g(\Phi(k)) \alpha_k  
=  \sum_{k=n}^{n+n_0 -1}   \barg_k \alpha_k  +\alpha_1 [ \barg_{n+n_0} - \barg_n]
\]

\item    Suppose that we have a sequence of functions satisfying, for some $B<\infty$,
\[
 g(k,\varble)\in\clG_L\,,\qquad  | g(k,z) - g(k-1,z) | \le  B \alpha_k  \,,\quad z\in\zstate\,,   \  k\ge 1
\]
Then,  (i) admits the extension 
\begin{equation}
\sum_{k=n}^{n+n_0 -1}   g(k, \Phi(k)) \alpha_k  
=    \barg(n)  \sum_{k=n}^{n+n_0 -1}    \alpha_k  +  \epsy(n,n_0; B, L)
\label{e:pre-ODE_approx}
\end{equation}
with $\barg(n) = \Expect_\varpi[ g(n, \Phi)]$, as defined in \eqref{e:QSAergodic}.   The error sequence depends only on $B$, $L$, and $\Phi(0)$,   and with these variables fixed satisfies  
\[
\lim_{n\to\infty} \epsy(n,n_0; B, L)   =0
\]
  \qed
\end{romannum}
\end{lemma}

\notes{
\bl{[Proof: to be inserted someday. assumptions are FAR too strong, but ok}
}

These approximations easily lead to ODE approximations for many of the algorithms introduced in this paper.  However, there is a slight mismatch:   for the batch algorithms,  we  obtain   recursions of the form 
\begin{equation}
\theta_{n+1}  = \theta_n    +    \alpha_{n+1}  \{  f(\theta_n, \Phi^{(r)}(n) )  + \epsy(\theta_n) \}
\label{e:SAbatch}
\end{equation}
where $\|  \epsy(\theta_n)  \| \le o(1)\| \theta_n \|$,  with $o(1)\to 0$,  and for $n\ge 1$, $r\ge 1$,
\[
 \Phi^{(r)}(n)  = (\Phi(n),\dots \Phi(n+r-1))
\]
Assumption (A$\bfqsaprobe$) implies the same ergodic theorems for this larger state process:
\begin{lemma}
\label[lemma]{t:e-chain_r}
Under (A$\bfqsaprobe$),
 the following limit exists for any continuous function $g\colon\zstate^r\to\Re$:  
\[
\barg\eqdef \lim_{\Nsam\to\infty }
%\barg_\Nsam(\Phi(0)) 
%= \lim_{\Nsam\to\infty}
 \frac{1}{\Nsam} 
 \sum_{k=1}^\Nsam  g(\Phi(k),\dots, \Phi(k+r-1)) 
\] 
\qed
\end{lemma}

The proof is straightforward since $g(\Phi(k),\dots, \Phi(k+r-1)) = g^{(r)} (\Phi(k))$ for some continuous function $g^{(r)}\colon\zstate\to\Re$, giving $\barg = \Expect_\varpi[ g^{(r)}(\Phi)] $.

Applying \eqref{e:BCQLSAa} of \Cref{t:BCQLSA}  we arrive at the approximate QSA recursion \eqref{e:SAbatch}, in which
\[
 f(\theta_n, \Phi^{(r)}(n) )   =       -    \half   \nabla  \clE_n^\infty(\theta_n) 
\]
The superscript indicates that  $r_n, \kappaBE_n ,  \kappa^+_n $ are replaced with their limits, giving
\[
\begin{aligned}
   \nabla  \clE_n^\infty(\theta_n)  
  =  2 \frac{1}{r} \sum_{k=T_n}^{T_n + r-1} \Big\{ &  -   \langle \mu, \psi^J \rangle  +  \Tdiff^\circ_{k+1}(\theta)   \bigl[  \psi^J_{(k+1)} -  \psi_{(k)}   \bigr] 
   \\
            & + \kappa^+   \{   J^\theta (x(k))  -  \uQ^{\theta} (x(k) )\}_+   \bigl[ \psi^J_{(k)}  -  \psi_{(k)}    \bigr]    \Bigr\}
\end{aligned} 
\]
The function $f$ is   Lipschitz continuous on $\Re^d\times\zstate$.   We let $\barf$ denote its limit:
\[
\barf(\theta) =  \lim_{\Nsam\to\infty}  \frac{1}{\Nsam} 
 \sum_{k=1}^\Nsam   f(\theta_k, \Phi^{(r)}(k) )   =  -    \half \nabla \Expect_\varpi[      \clE_n^\infty(\theta)  ]\,,\qquad \theta\in\Re^d
\]
The ODE of interest is then defined by \eqref{e:ODEideal} with this $\barf$:
\[
\ddt \odestate_t= \barf(\odestate_t) 
\]
Recall the definition of   the ``sampling times'' $\{ \Hor_n \}$ provided in
\eqref{e:ODE_TimeInterval},   and the definition of the ODE approximation in
\eqref{e:ODE_Approximation_Defined}.

\begin{lemma}
\label[lemma]{t:batchBdd}
Under the assumptions of 
 \Cref{t:ConvexDQNconverges},  
there is a fixed $\barsigma_\theta<\infty $ such that for each initial condition $\Phi(0)$,
\[
\limsup_{n\to\infty} \|\theta_n \| \le \barsigma_\theta
\]
\end{lemma}

\begin{proof}
This is established via the scaled ODE technique of   \cite{bormey00a,bor20a,rambha17,rambha18}.   For each $n\ge 0$ denote $s_n = \max(1, \| \ODEstate_{\Hor_n } \|)$,    and 
\[
\ODEstate^n_t = \frac{1}{s_n} \ODEstate_{ t},\quad t\ge \Hor_n \,,\qquad  f^n(\theta,z) =  \frac{1}{s_n} f(s_n \theta , z) 
\,,\qquad  \barf^n(\theta) =  \frac{1}{s_n} \barf(s_n \theta) 
\]
\Cref{t:LLNuni,t:e-chain_r}  justify the approximation
\begin{equation}
\ODEstate^n_{\Hor_n + t} 
	= \ODEstate^n_{\Hor_n} + \int_{\Hor_n}^{\Hor_n + t}    \barf^n (\ODEstate^n_r) \, dr  + \epsy^n_t \,, \qquad 0\le t \le \Hor 
\label{e:BorkarMeynApprox}
\end{equation}
where $\sup\{ \| \epsy^n_t  \|  :  t\le \Hor \}  $ converges to zero as $n\to\infty$.  

For the proof here we require a modification of the definition of the ODE approximation,  since we are approximating the scaled process $\bfODEstate^n$ rather than $\bfODEstate$.  Denote by
 $\{ \odestate_t^n : t\ge \Hor_n\}$     the solution to the ODE  \eqref{e:ODEideal} with initial condition  
\[
 \odestate^n_{\Hor_n}  = \ODEstate^n_{\Hor_n} =    s_n^{-1}  \ODEstate_{\Hor_n } \,, \qquad s_n = \max(1, \| \ODEstate_{\Hor_n } \|)
\]
An application of the Bellman-Gronwall Lemma is used to obtain the ODE approximation: 
\[
\lim_{n\to\infty}  \sup_{\Hor_n \le t\le \Hor_n +\Hor }  \| \odestate^n_t -  \ODEstate^n_t   \|  =0
\]
Uniformity in this step depends on the fact that $\|  \odestate^n_{\Hor_n} \|\le 1$. 
Next, recognize that the dynamics of the scaled ODE are approximated by gradient descent with cost function equal to zero:
\[
\lim_{s\to\infty} \barf^s(\theta)  = - \half \nabla \bar\clE^0(\theta) 
\]
in which   
\begin{equation}
\bar\clE^0(\theta) = \Expect_\varpi [  \clE_n^{\upvarepsilon,0}(\theta)  +  \kappa^+ \clE_n^+(\theta)  ]
\label{e:clEzero}
\end{equation} 
with $\clE_n^+$ defined in  \eqref{e:ConvexQ_loss+}, but with the cost function removed in  \eqref{e:ConvexQ_lossBE}:
\[
\clE_n^{\upvarepsilon,0}(\theta)  =  \{ -  Q^\theta (x(k),u(k) )    +   J^{\theta} (x(k+1))  \bigr\}^2  
\]
The function $\bar\clE^0$ is strongly convex, with unique minimum at $\theta =\Zero$.  From this we obtain a drift condition:    for $\Hor$ chosen sufficiently large we have for some constant $b_\theta$,
\[
\|  \odestate^n_{\Hor_{n+1}}  \|  \le  \fourth \|  \odestate^n_{\Hor_{n}}  \|  \qquad \text{whenever $ \|  \odestate^n_{\Hor_{n}}  \|   \ge b_\theta$}
\]
We then obtain a similar contraction for the unscaled recursion:    for all $n\ge 1$ sufficiently large,
 \[
\|  \ODEstate_{\Hor_{n+1}}  \|  \le  \half \| \ODEstate_{\Hor_{n}}  \|  \qquad \text{whenever $ \|  \ODEstate_{\Hor_{n}}  \|   \ge b_\theta$}
\]
This implies that $\{  \ODEstate_{\Hor_{n}} \}$ is a bounded sequence, and Lipschitz continuity then implies   boundedness of
$ \{  \ODEstate_t  :  t\ge 0 \}$ and hence also $ \{  \theta_k : k\ge 0\}$.
\end{proof}

\begin{proof}[Proof of \Cref{t:ConvexDQNconverges}]
The boundedness result  \Cref{t:batchBdd} was the hard part.   Convergence of the algorithm follows from boundedness, combined with the ODE approximation techniques of \cite{bor20a} (which are based on arguments similar to those leading to   \eqref{e:BorkarMeynApprox}).
\end{proof}

\begin{proof}[Proof of \Cref{t:ConvexQ_limit_cor}] 
We require a representation similar to \eqref{e:SAbatch} in order to establish boundedness of $\{\theta_n\}$, and then apply standard stochastic approximation arguments.
For this we write $\Theta = \{\theta \in\Re : Y \theta\le \Zero \}$ for a matrix $Y$ of suitable dimension.  This is possible under the assumption that the constraint set is a polyhedral cone.    We then obtain via a Lagrangian relaxation of 
\eqref{e:ConvexQalmostSAcor_theta},
\[
\Zero = \nabla_\theta \Bigl \{  -   \langle \mu,J^\theta \rangle      + \kappaBE \clEBE_n(\theta)   
		-  \lambda_n^\transpose  \zBE_n(\theta)  
		+\gamma_{n+1}^\transpose Y \theta
		+  \frac{1}{\alpha_{n+1}} \half     \| \theta -\theta_n \|^2 \Bigr\} \Big|_{\theta = \theta_{n+1}  }
\]
where $\gamma_{n+1}\ge \Zero$ satisfies the KKT condition $\gamma_{n+1}^\transpose Y \theta_{n+1} =0$\footnote{We are taking the  gradient of a quadratic function of $\theta$,   so there is a simple formula for the Lagrange multiplier $\gamma_{n+1}$.    }.
   We thus obtain the desired approximation:  following the argument in \Cref{t:BCQLSA},
\begin{equation}
 \theta_{n+1}  =  \theta_n  -      \alpha_{n+1} \bigl\{ -    \langle \mu, \psi^J \rangle
      + \kappaBE  \nabla  \clEBE_n(\theta_n)   
		-  \lambda_n^\transpose \nabla \zBE_n(\theta_n )    
		+ Y^\transpose \gamma_{n+1}   + \epsy_{n+1}  \bigr\} 
\label{e:cor_Primal-Dual_ODE}
\end{equation}
where $ \| \epsy_{n+1} \| = O( \alpha_{n+1} \|\theta_n\|)$.

Using the definition of the interpolated process $\bfODEstate$ defined above \eqref{e:ODE_Approximation_Defined}, we obtain the approximation 
\begin{equation}
\begin{aligned}
\ODEstate_{\Hor_n + t}   &=  \ODEstate_{\Hor_n }  - \int_{\Hor_n}^ {\Hor_n + t} \Bigl\{ -    \langle \mu, \psi^J \rangle
      + \kappaBE  \nabla  \barclEBE(\ODEstate_r )   
		-  [\nabla \barzBE(\ODEstate_r ) ]^\transpose  \uplambda_r    \Bigr\}  \, dr    - \int_{\Hor_n}^ {\Hor_n + t}   Y^\transpose d \upgamma_r   + e_\theta(t,n)
   \\
\uplambda_{\Hor_n + t}   &=  \uplambda_{\Hor_n }   -   \int_{\Hor_n}^ {\Hor_n + t}  \barzBE(\ODEstate_r)  \, dr
    + \int_{\Hor_n}^ {\Hor_n + t}  d \upgamma_r^+ -  \int_{\Hor_n}^ {\Hor_n + t}  d \upgamma_r^-   + e_z(t,n)
\end{aligned} 
\label{e:ODE_CQL_PrimalDual}
\end{equation}
where $\{\upgamma_r\}$ is a vector valued processes with non-decreasing components,  and  $\{\upgamma_r^\pm\}$ are scalar-valued non-decreasing processes (arising from the projection in \eqref{e:ConvexQalmostSAcor_l}).    The error processes satisfy
\[
\sup\{ \| e_\theta(t,n) \| + \| e_t(t,n)\| :  0\le t\le T \}  = o(1) \| \ODEstate_{\Hor_n }  \|
\]
for any fixed $T$

The proof of boundedness of $\{\theta_n \}$ is then obtained exactly exactly as in \Cref{t:batchBdd}: the ``large state'' scaling results in the approximation by a linear system with projection:   
\[
\odestate^n_{\Hor_n + t}   =  \odestate^n_{\Hor_n }  - \int_{\Hor_n}^ {\Hor_n + t}  M \odestate^n_r \, dr   - \int_{\Hor_n}^ {\Hor_n + t}   Y^\transpose d \Gamma_r 
\]
where $M \theta = \nabla \bar\clE^0(\theta) $, and the mean quadratic loss is defined in \eqref{e:clEzero} with $  \kappa^+=0$. 
Note that $ \odestate^n_{\Hor_n } = s_n^{-1}  \ODEstate_{\Hor_n } $, and  as in the proof of  \Cref{t:batchBdd},   we have replaced the vector field $  \barf^n $ used there with its limit $  \barf^\infty$.

  The process $\bfGamma$ has non-decreasing components, satisfying
 for all $t$,
\[
0 =  \int_{\Hor_n}^ {\Hor_n + t}   \{ Y  \odestate^n_r\}^\transpose  d \Gamma_r 
\]
Stability of the ODE follows using the Lyapunov function $V(\theta) = \half \|\theta\|^2$:  for $r\ge \Hor_n$, 
\[
dV( \odestate^n_r   ) = {\odestate^n_r} ^\transpose d\odestate^n_r = 
			-  { \odestate^n_r} ^\transpose M \odestate^n_r \, dr  
			- { \odestate^n_r }^\transpose Y^\transpose d \Gamma_r   =  -  { \odestate^n_r} ^\transpose M \odestate^n_r \, dr 
\]
We have $M>0$, so it follows that $\odestate^n_r\to 0$ exponentially fast as $r\to\infty$.

The arguments in \Cref{t:batchBdd} then establish boundedness of $\{ \theta_n \}$, and then standard arguments imply an ODE approximation without scaling, where
the ODE approximation of \eqref{e:ODE_CQL_PrimalDual} is the same set of equations, with the error processes removed:
\[
\begin{aligned}
\odestate_{\Hor_n + t}   &=  \ODEstate_{\Hor_n }  - \int_{\Hor_n}^ {\Hor_n + t}  \nabla_\theta \barL(\odestate_r , \Uplambda_r )    - \int_{\Hor_n}^ {\Hor_n + t}   Y^\transpose d \Upgamma_r    
   \\
\Uplambda_{\Hor_n + t}   &=  \uplambda_{\Hor_n }   -   \int_{\Hor_n}^ {\Hor_n + t}    \nabla_\lambda \barL(\odestate_r , \Uplambda_r )    \, dr
    + \int_{\Hor_n}^ {\Hor_n + t}  d \Upgamma_r^+ -  \int_{\Hor_n}^ {\Hor_n + t}  d \Upgamma_r^-    
\end{aligned} 
\]
where $\barL$ is defined in \eqref{e:SA_saddle}.  For this we adapt analysis in   \cite[Section~1.2]{bor20a}, where a similar ODE arises.

The stability proof of  \cite[Section~1.2]{bor20a}  amounts to showing $V(t) = \half \{ \| \odestate_t  - \theta^\ocp\|^2 + \| \Uplambda_t- \lambda^\ocp \|^2 \}$ is a Lyapunov function:
\begin{equation}
\int_{t_0}^{t_1}
 V(t) \, dt < 0\,,  \quad \textit{whenever $ (\odestate_t , \Uplambda_t )  \neq  ( \theta^\ocp, \lambda^\ocp)$,    }
\label{e:PD_negDrift}
\end{equation}
and any $t_1>t_0\ge \Hor_n$.
The arguments there are for an unreflected ODE, but the proof carries over to the more complex setting:   
 \[
\begin{aligned}
dV(t)  &= 
		  \langle  \odestate_t  - \theta^\ocp ,    d   \odestate_t \rangle 
		+   \langle  \Uplambda_t  - \lambda^\ocp ,  d   \Uplambda_t \rangle				
   \\
            &   \le   		  \langle  \odestate_t  - \theta^\ocp ,     d   \odestate_t^0 \rangle  
		+   \langle  \Uplambda_t  - \lambda^\ocp ,  d   \Uplambda_t^0 \rangle	
\end{aligned} 
\]
where the super-script ``${}^0$'' refers to differentials with the reflections removed:
\[
 d   \odestate_t^0  =  -  \nabla_\theta \barL(\odestate_t ,  \Uplambda_t )   \,dt  \,,  \qquad d   \Uplambda_t^0 =\nabla_\lambda \barL(\odestate_t ,  \Uplambda_t )  \, dt
\]
The inequality above is justified by the reflection process characterizations:
\[
\begin{aligned}
   \odestate_t^\transpose Y^\transpose d \Upgamma_t  &=0\,,\quad   \Uplambda_t^\transpose  d \Upgamma_t^+ = 0   \,,\quad   \Uplambda_t^\transpose  d \Upgamma_t^- =( \lambda^{\text{max}})^\transpose d \Upgamma_t^- 
   \\ 
{ \theta^\ocp }^\transpose Y^\transpose d \Upgamma_t  &\le    0\,,\quad   {\lambda^\ocp }^\transpose  d \Upgamma_t^+ \ge  0   \,,\quad   {\lambda^\ocp }^\transpose  d \Upgamma_t^- \le( \lambda^{\text{max}})^\transpose d \Upgamma_t^- 
  \end{aligned} 
\]
From this we obtain
\[
\begin{aligned}
dV(t)  & \le   		- \langle  \odestate_t  - \theta^\ocp ,    \nabla_\theta \barL(\odestate_t ,  \Uplambda_t ) \rangle \, dt
		+   \langle  \Uplambda_t  - \lambda^\ocp ,    \nabla_\lambda \barL(\odestate_t ,  \Uplambda_t )\rangle	\, dt
		\\
		&\le 
		- [  \barL(\odestate_t\, ,  \lambda^\ocp)  -  \barL(\theta^\ocp\, ,  \lambda^\ocp)   ] \,  dt  
		-
		[    \barL(\theta^\ocp\, ,  \lambda^\ocp)  -  \barL(\theta^\ocp\, ,  \Uplambda_t)   ]  \, dt          
\end{aligned} 
\]
where the second inequality follows from convexity of $\barL$ in $\theta$, and linearity in $\lambda$.
This establishes the desired negative drift \eqref{e:PD_negDrift}.
\end{proof}

\begin{proof}[Proof of \Cref{t:dumbDQN}]     Let $L_Q>0$ denote the Lipschitz constant for  $\nabla  Q^\theta$.   It follows that by increasing the constant we obtain two bounds:  for all $\theta,\theta'\in\Re^d$,
\begin{equation}
\| \nabla  Q^\theta (x,u) -\nabla  Q^{\theta'} (x,u) \|  \le L_Q \| \theta - \theta'\| \,,\qquad \|    Q^\theta (x,u) -   Q^{\theta'} (x,u) \|  \le L_Q \| \theta - \theta'\| (1+\|\theta\| + \|\theta'\| )
\label{e:Qlip}
\end{equation}
The second bound follows from the identity $Q^{r\theta}(x,u) = Q^{\theta}(x,u) +\int_1^r \theta^\transpose \nabla Q^{s\theta}(x,u) \, ds$,  $r\ge 1$  (and is only useful to prove that DQN is convergent -- recall that convergence is assumed in the proposition).

The proof begins with an extension of \Cref{t:BCQLSA}, to express the DQN recursion as something resembling stochastic approximation.  
This will follow from the first-order condition for optimality:
\[
\begin{aligned}
\Zero& = \nabla    \Bigl \{       \clEBE_n(\theta)  +   \frac{1}{\alpha_{n+1}}       \| \theta -\theta_n \|^2\Bigr\}
  \Big|_{\theta=\theta_{n+1}}  
\\
&  =  - 2 \frac{1}{r}  \sum_{k=T_n}^{T_{n+1}-1}  \bigl[    -  Q^\theta (x(k),u(k) )   +   c(x(k),u(k))  +   \uQ^{\theta_n} (x(k+1))  \bigr] \nabla  Q^\theta (x(k),u(k) )  \Big|_{\theta=\theta_{n+1}}  
\\
& \qquad +   \frac{1}{\alpha_{n+1}} \bigl(  \theta_{n+1}  -\theta_n \bigr)
\end{aligned} 
\]
It is here that we apply the Lipschitz continuity bounds in \eqref{e:Qlip}.   This allows us to   write
\[
2 \frac{1}{\alpha_{n+1}} \bigl(  \theta_{n+1}  -\theta_n \bigr)  =  2 \frac{1}{r}  \sum_{k=T_n}^{T_{n+1}-1}  \bigl[    -  Q^{\theta_k} (x(k),u(k) )   +   c(x(k),u(k))  +   \uQ^{\theta_k} (x(k+1))  \bigr]  \elig_k  +   \epsy_n
\]
where $\|\epsy_n\|   =    O(\alpha_n )$  under the assumption that $\{\theta_n\}$ is bounded.  
This brings us to a representation similar to  \eqref{e:WatkinsFunctionApproximation}:
\[
  \theta_{n+1}  =\theta_n    + 
 \alpha_{n+1}  \Bigl\{
    \frac{1}{r}  \sum_{k=T_n}^{T_{n+1}-1}  \bigl[    -  Q^{\theta_k} (x(k),u(k) )   +   c(x(k),u(k))  +   \uQ^{\theta_k} (x(k+1))  \bigr]  \elig_k  +   \epsy_n \Bigr\}
\]
Stochastic approximation/Euler approximation arguments then imply the ODE approximation,  and the representation of the limit.  
\end{proof}

 \end{document}